\documentclass[a4paper, 11pt]{article}
\usepackage{mathrsfs,amssymb,amsmath}

\usepackage[top=1.31in,bottom=1.29in,left=1.1in,right=0.9in]{geometry}

\usepackage{amsthm}
\numberwithin{equation}{section}\theoremstyle{definition}

 \swapnumbers

 \newtheorem{Theorem}[equation]{Theorem}
 \newtheorem{Prop}[equation]{Proposition}
 
 \newtheorem{Lemma}[equation]{Lemma}
 \newtheorem{Cor}[equation]{Corollary}
 
 \newtheorem{Defn}[equation]{Definition}
 \newtheorem{Example}[equation]{Example}
 \newtheorem{Remark}[equation]{Remark}

\makeatletter
\def\enumerate{\begingroup\ifnum\@enumdepth>3\@toodeep\else
      \advance\@enumdepth\@ne
      \edef\@enumctr{enum\romannumeral\the\@enumdepth}%
      \topsep\z@\parskip\z@
      \list{\csname label\@enumctr\endcsname}
        {\@nmbrlisttrue\let\@listctr\@enumctr
         \parsep\z@\itemsep\z@\topsep\z@
         \setcounter{\@enumctr}{0}
         \def\makelabel##1{\hss\llap{\rm ##1}}
       }\fi}

\makeatother
\def\({\big(}
\def\){\big)}

\def\C{\mathbb C}
\def\N{\mathbb N}

\def\Z{\mathbb Z}
\def\F{\mathbb F}

\def\E{\mathcal E}

\def\c{\mathbf C}

\def\O{\mathcal O}

\def\T{\mathcal T}
\def\cL{\mathcal L}
\def\cR{\mathcal R}

\def\p{\mathfrak p}

\DeclareMathOperator{\Hom}{Hom}

\DeclareMathOperator{\Ind}{Ind}

\DeclareMathOperator{\row}{row}

\DeclareMathOperator{\End}{End}

\DeclareMathOperator{\Pstab}{Pstab}

\DeclareMathOperator{\main}{main}
\DeclareMathOperator{\verge}{verge}
\DeclareMathOperator{\suppl}{suppl}
\DeclareMathOperator{\Irr}{Irr}

\begin{document}
\setlength{\abovedisplayskip}{4pt}
\setlength{\belowdisplayskip}{4pt}
\parindent=0pt


\title{Irreducible constituents of minimal degree in supercharacters of the finite unitriangular groups}

\author{Richard Dipper, Qiong Guo\\ \\Institut f\"{u}r Algebra und Zahlentheorie\\ Universit\"{a}t Stuttgart, 70569 Stuttgart, Germany
\\ \scriptsize{E-mail: richard.dipper@mathematik.uni-stuttgart.de, guo.qiong@mathematik.uni-stuttgart.de}
\setcounter{footnote}{-1}\footnote{\emph{Date:} December 12th, 2013.}
\setcounter{footnote}{-1}
\setcounter{footnote}{-1}}

\date{}


\maketitle

\begin{abstract}
Let $q$ be a prime power and $U$ the group of lower unitriangular matrices of order $n$ for some natural number $n$. We give a lower bound for the degrees of irreducible constituents of Andr\'{e}-Yan supercharacters and classify the supercharacters having constituents whose degree assume this lower bound. Moreover we show that the number of distinct irreducible characters of $U$ meeting this condition is a polynomial in $(q-1)$ with nonnegative integral coefficients and exhibit monomial sources for those.

\end{abstract}

\section{Introduction}
Let $p$ be a prime,  $q$  a power of $p$, $\F_q$ the finite field with $q$ elements and $U_n(q)$ ($n\in \N$) the group of lower unitriangular $(n\times n)$-matrices with entries in $\F_q$. Thus $U$ is a $p$-Sylow subgroup of the full general linear group $GL_n(q)$. It is known that determining the conjugacy classes of $U$ for all $n$ and $q$ is a wild problem. Even finding their number as function $C(n,q)$ of $q$ and $n$ and hence the number of distinct irreducible complex characters is still an open problem. A longstanding conjecture contributed to  G. Higman \cite{higman} states that $C(q,n)$ should be a polynomial in $q$ with integral coefficients depending only on $n$, not on $ q$. G. Lehrer refined this by conjecturing that the number of pairwise distinct irreducible characters of $U$  of a fixed degree $q^c, c\in \Z_{\geqslant 0}$, should be a polynomial in $q$ with integral coefficients \cite{lehrer}, which more recently was refined once more by Isaacs who conjectured that these polynomials should be actually polynomials in $(q-1)$ with nonnegative integral coefficients \cite{Isaacs2}. This, of course, uses the fact that the degrees of the complex irreducible characters of the $p$-group $U$ are powers of $q$ not just of $p$, by a result of Huppert \cite{huppert}, which is actually true for $\F_q$-algebra groups by a theorem of Isaacs \cite{Isaacsq}.
The Andr\'{e}-Yan supercharacter theory \cite{andre1}, \cite{yan} provides an approximation to the problem of classifying the irreducible complex characters of $U$. A supercharacter theory for some finite group $G$ consists of a set partition of the collection of conjugacy classes, the unions of the parts called superclasses, and a set of pairwise orthogonal complex characters, called supercharacters such that every irreducible complex character of $U$ occurs as constituent in precisely one supercharacter. Moreover superclasses and supercharacters are in 1-1 correspondence and supercharacters are constant on superclasses.

\smallskip

An $\F_q$-algebra group $G$ is of the form $G=\{1+x\,|\, x\in J(A)\}$ for some finite dimensional $\F_q$-algebra $A$ with Jacobson radical $J(A)$. Taking for $A$ the $\F_q$-algebra of lower triangular matrices, $V=J(A)$ is the nilpotent $\F_q$-algebra of strictly lower triangular matrices and $U=\{1+x\,|\,x\in J(A)\}$ is indeed an $\F_q$-algebra group. Moreover $U$ acts on $V=J(A)$ by left and right multiplication and hence on the set of linear complex characters $\hat V$ of the additive group $(V, +)$. The map $f: U\rightarrow V: u\mapsto u-1\in V$ induces an 1-cocycle $\alpha: \hat V\times U\longrightarrow \C ^*: (\chi, u)\mapsto \chi(u^{-1}-1)$ providing a right monomial action of $U$ on $\C \hat V$ with monomial basis $\hat V$. There is a similar left hand side construction for a monomial action of $U$ on $\hat V$ from the left. With this action from both sides $\C \hat V$ becomes an $U$-$U$-bimodule which is isomorphic to the regular bi-representation $_{\C U} \C U _{\C U}$. Each biorbit of $\hat V$ under the action of $U$ is a union of right orbits. It turns out that all right orbits in a biorbit induce isomorphic right modules, and any two right orbits being contained in different biorbits afford orthogonal characters. The different (and hence orthogonal) characters afforded by all the right orbits are the Andr\'{e}-Yan supercharacters.
These can be described combinatorially. We show that  there is for each right orbit a lower bound for the degrees of irreducible constituents occurring in the representation of $\C U$ on $\C \O$ for an orbit of $\hat V$ under the right action of $U$. We say the irreducible $\C U$-module $S$ has minimal dimension, if $\dim_\C S$ assumes this lower bound in the right orbit module of which it is an irreducible constituent.

\smallskip

Inspecting the endomorphism rings of right $U$-orbit modules we obtain a completely combinatorial necessary and sufficient condition for supercharacters having irreducible character of minimal degree as constituents. This is our first main result.
Moreover we show, those have multiplicity one in their supercharacters and show that there are $q^c$, $c\in \N$ many irreducible constituents of minimal degree in supercharacters, where $c\in \N$ is determined combinatorially. As a consequence we obtain that the number of distinct irreducible characters of $U$ of minimal degree in their supercharacters (of degree $q^d, d\in \Z_{\geqslant 0}$ fixed)  is a polynomial in $(q-1)$ with nonnegative integral coefficients.

\smallskip

By a theorem of Halasi \cite{halasi}, every irreducible character $\mu$ of $U$ is induced from a linear character $\lambda: H\rightarrow \C ^*=\C\setminus \{0\}$ for some $\F_q$-algebra subgroup $H$ of $U$. We call the pair $(H, \lambda)$ a monomial source of $\mu\in \Irr(U)$. Our second main result determines monomial sources of irreducible characters of $U$ of minimal degree in their supercharacters, (\ref{6.11}). These are in fact linear characters of certain pattern subgroups of $U$ and hence irreducible characters of minimal degree in their supercharacters are well-induced in the sense of Evseev \cite{Evseev}.

\smallskip

In her doctoral thesis \cite{guo} the second author determined the irreducible $U$-constituents of the permutation module of $GL_n(q)$ on the cosets of a maximal parabolic subgroup of $GL_n(q)$. In a forthcoming paper we shall show that these irreducible $\C U$-modules are precisely the irreducible constituents of minimal dimension in a certain combinatorially determined subclass of $U$-orbit modules having irreducible constituents of minimal degree.

\medskip

We now fix some notation which is used throughout this paper.
 We identify the set
$\Phi=\{(i,j)\,|\,1\leqslant i, j \leqslant n, i\not=j\}$ with
the standard root system of $G$ where
$\Phi^+=\{(i,j)\in \Phi \,|\,i>j\}$, $\,\Phi^-=\{(i,j)\in \Phi
\,|\,i<j\}$ are the positive respectively negative roots with
respect to the basis $\Delta=\{(i+1,i)\in
\Phi^+\,|\,1\leqslant i \leqslant n-1\}$ of $\Phi$.
A subset $J$ of $\Phi$ is {\bf closed} if $(i,j),(j,k)\in J, (i,k)\in \Phi$ implies
$(i,k)\in J.$
For $1\leqslant i, j \leqslant n$ let $\epsilon_{ij}$ be the
$n\times n$-matrix $g=(g_{ij})$ over $\F_q$, with $g_{ij}=1$ and
$g_{kl}=0$ for all $1\leqslant k,l \leqslant n$ with
$(k,l)\not=(i,j).$ Thus $\{\epsilon_{ij}\,|\, 1\leqslant i, j
\leqslant n \}$ is the natural basis
of the $\F_q$-algebra $M_n(\F_q)$ of $n\times n$-matrices with
entries in $\F_q.$
For $1\leqslant i, j \leqslant n,$\, $i\not=j$ and $\alpha\in
\F_q$, let $x_{ij}(\alpha)=E_n+\alpha \epsilon_{ij},$ where $E_n$
is   the $n\times n$-identity matrix. Then
$X_{ij}=\{x_{ij}(\alpha)\,|\,\alpha\in \F_q\}$ is the {\bf root
subgroup} of $G$ associated with
the root $(i,j)\in \Phi,$ and is isomorphic to the additive
group $(\F_q,+)$ of the underlying field $\F_q$, hence is in
particular abelian. Moreover $U=\langle
x_{ij}(\alpha)\,|\,1\leqslant j<i\leqslant n, \, \alpha\in
\F_q\rangle$ is the unitriangular subgroup of $G=GL_n(q)$
consisting of all lower triangular matrices with ones on the
diagonal. It is well known that for a closed subset
$J$ of $\Phi^+$, the set
$
U_J=\{u\in U\,|\,u_{ij}=0,\,\forall\, (i,j)\notin J\}$ is the
subgroup of $U$ generated by $X_{kl},\, (k,l)\in J$ and if we choose any linear ordering on $J$ then
$U_J=\{\prod_{(i,j)\in
J}x_{ij}(\alpha_{ij})\,|\,\alpha_{ij}\in \F_q\}$, where the
products are given in the fixed linear ordering.

\section{$U$-Supercharacters}
In \cite{yan} Yan constructed a basis of $\C U$, called Fourier basis, such that $U$ acts monomially (from both left and right) on it. We shall give here a brief overview on Yan's construction setting it up in a notation more suitable for our work here than the original one used by Yan. We shall use an approach introduced by Markus Jedlitschky in his thesis  \cite{markus} which produces Yan's Fourier basis. We begin with a very general setting.

For the moment let $G$ be an arbitrary group and $(V,+)$ be an abelian group on which $G$ acts as group of automorphisms, the action denoted by right multiplication. Let $K$ be a field. Then $G$ acts on $K^V$, the set of functions from $V$ to $K$ by
\begin{equation}\label{2.1}
 (\tau\ldotp g)(A)=\tau(A g^{-1}) \text{ for } \tau\in K^V, g\in G, A\in V. 
\end{equation}
In particular, the subset $\hat V=\Hom(V, K^*)$ of linear characters of $V$ in $K$ of $K^V$ is $G$-invariant, since for $\chi\in \hat V$, $A, B\in V$ and $g\in G$ we have
\begin{eqnarray*}
 (\chi\ldotp g)(A+B)&=&\chi((A+B))g^{-1}=\chi(Ag^{-1}+Bg^{-1})\\
&=& \chi(Ag^{-1})\chi(Bg^{-1})=(\chi\ldotp g)(A) \chi\ldotp g(B)
\end{eqnarray*}
proving that $\chi\ldotp g$ is again a linear character on $V$. Suppose now that $f: G\longrightarrow V$ is a (right sided) 1-cocycle, i.e. we have
\begin{equation}\label{cocycle}
 f(xg)=f(x)g+f(g) \quad\forall\, x,g\in G.
\end{equation}
We define $\alpha: \hat V \times G \longrightarrow K^*: (\chi, g)\mapsto \chi(f(g^{-1}))$ for $\chi\in \hat V, g\in G,$ then the following holds:
\begin{Theorem}[Jedlitschky \cite{markus}]\label{jed}
Let $K\hat V$ be the $K$-vector space with basis $\hat V$. Then $K \hat V$ becomes a $KG$-module with monomial basis $\hat V$, where the new action denoted by ``$*$'' of $G$ on $\hat V$ is given as
$$
\chi * g=\alpha(\chi,g) \chi\ldotp g=\chi(f(g^{-1}) )\chi \ldotp g
$$ 
for $\chi\in \hat V, g\in G$.
\end{Theorem}
The fact that $f$ satisfies (\ref{cocycle}) ensures, that the $*$-action on  $K \hat V$ is indeed compatible with the multiplication in $G$, that is we have
$$
(\chi* g)* h=\chi *(gh) \quad \forall\, \chi\in \hat V, g, h\in G.
$$
For simplicity, we replace the $*$ in \ref{jed} by the standard notation and write $\chi g$ instead of $\chi* g$. Obviously, if $G$ acts on $V$ from the left as group of automorphisms,  replacing $f$ by an left 1-cocycle $\tilde f: G\longrightarrow V$ satisfying $\tilde f(gx)=g\tilde f(x)+\tilde f(g)$, we can define analogously the left $KG$-module $K\hat V$ with monomial basis $\hat V$. Moreover if $G$ acts on $V$ from both sides such that $(gA)h=g(Ah)$ for all $g, h\in g$ and $A\in V$, and if $f$ is a right and left 1-cocycle, $K\hat V$ becomes an $KG$-bimodule with monomial action of $G$ on the base $\hat V$ from both sides. $K\hat V$ decomposes into a direct sum of $KG$-modules (from right, left, both sides) of $G$ on $\hat V$ under the dot permutation action ``$\ldotp$'' in (\ref{2.1}).

We apply this to the special case, where $G=U=U_n(q)$ is the group of lower unitriangular $n\times n$-matrices with entries in the filed $\F_q$ with $q$ a prime power. Let $V$ be the set of nilpotent strictly lower triangular $n\times n$-matrices over $\F_q$, thus $V=\{u-1\,|\,u\in U\}$. Then $V$ is the Lie algebra of $U$ and in particular an abelian group under addition of matrices. $U$ acts on $V$ by right- and left multiplication group of automorphisms and both action commute by associativity of matrix multiplication. Henceforth we take $K=\C$ and choose once for all a non trivial character $\theta:(\F_q,+)\longrightarrow \C^{*}$. Moreover the $V^*=\Hom(V, \F_q)$ has a basis given by the coordinate functions $\xi_{ij}: V\longrightarrow \F_q: A=A_{ij}\in \F_q$, where $A_{ij}$ denotes the entry of the matrix $A$ at position $(i,j)$, $1\leqslant i,j \leqslant n$. For a matrix $B\in V$, we define the linear $\C$-character $\chi_{_B}\in \hat V=\Hom((V,+),\C^*)$ to be 
$$
\chi_{_B}=\theta\circ (\sum_{i,j} B_{ij}\xi_{ij}).
$$
Then $\hat V=\{\chi_{_B}\,|\,B\in V\}$.
Now consider the map $f: U\longrightarrow V: u\mapsto u-1$. It can be seen easily that $f$ is a two sided 1-cocycle from $U$ to $V$, thus we may apply theorem \ref{jed} to turn the $\C$-space $\C \hat V$ into an $\C U$-bimodule, where $U$ acts on both sides on the basis $\hat V$ of $\C\hat V$ monomially.

Indeed the $\C U$-bimodule $\C \hat V$ is isomorphic to the regular $\C U$-bimodule $_{ \C U}\C U_{ \C U}$. This can be shown using the fact that $f: U\longrightarrow V: u\mapsto u-1$ is a bijection. Moreover, since $V$ is a finite group, the $\C$-vector space $\C^V$ of functions from $V$ to $\C$ is isomorphic to the group algebra $\C V$ as $\C V$-module. Here for any field $\C $ the action of $V$ on $\C^V$ from the right is given in (\ref{2.1}) and similarly from the left by setting
$$
A \ldotp \tau(B)=\tau(-A+B), \text{ for } \tau\in \C^V, A, B\in V.
$$
(Recall that addition in the abelian group $V$ turns into multiplication in the group algebra $\C V$). The isomorphism $\C^V\longrightarrow \C V$ then is given by evaluation of $\tau\in \C^V$:
$$
\tau\mapsto \sum_{A\in V} \tau(A) A.
$$
In particular, for $A\in V$, the linear character $\chi_A\in \hat V \leqslant \C^V$ is mapped to
\begin{equation}
 \chi_{_A }\mapsto \sum_{B\in V}\chi_{_A}(B) B =|V|e_{_{-A}}\in \C V,
\end{equation}
 where $e_{_{-A}}$ is the primitive idempotent in $\C V$ associated with $-A\in V$. Thus \ref{jed} yields an action of $\C U$ on $\C V$ such that $U$ acts on the basis $\{e_{_A}\,|\,A\in V\}$ of $\C V$ monomially.

Next we describe the action of $U$ on $\hat V$. Recall that for $A\in V, u\in U$, the linear character $\chi_A\ldotp u \in \hat V$ is given by
\begin{equation}
( \chi_{_A} \ldotp u ) (B)=\chi_{_A}(Bu^{-1})\quad\forall\, B\in V.
\end{equation}
Since $\chi_A\ldotp u\in \hat V$ there must be a $C\in V$ such that $\chi_A\ldotp u=\chi_C$. In order to describe $C$ for given $A$ and $u$ we let  \,$\bar{} : M_{n\times n}(\F_q)\longrightarrow V: A \mapsto \overline{A}\in V$ the natural projection. Recall that $\epsilon_{ij}\in M_{n\times n}(\F_q)$ denotes the $(i,j)$-th matrix unit. For $A\in M_{n\times n}(\F_q)$, we denote by $A_{ij}\in \F_q$  the entry of $A$ at position $(i,j)$, for $1\leqslant i, j\leqslant n$. Thus $A=\sum_{1\leqslant i, j\leqslant n}A_{ij}\epsilon_{ij}$ and $\overline A =\sum_{1\leqslant  j<i\leqslant n}A_{ij}\epsilon_{ij}$. For $u\in U$ let $u^t$ denote the transposed matrix (an upper unitriangular matrix). Then we have the following:
\begin{Lemma}[Yan]\label{2.6}
 Let $A\in V$ and $u\in U$. Then $\chi_{_A} \ldotp u=\chi_{_C}$, where $C=\overline{Au^{-t}}$, setting $u^{-t}=(u^{-1})^t$.
\end{Lemma}

By general theory the idempotent $e_{_A}$ in $KV$ affording the linear character $\chi_{_A}$ of $V$ is given as:
$$e_{_A}=\frac{1}{|V|}\sum_{B\in V}\overline{\chi_{_A}(B)}B,$$
where $\bar{} : \C\rightarrow \C: z\mapsto \bar{z}$ denotes complex conjugation.

We write $[A]=e_{_{A}}$ and illustrate this idempotent by a triangle, omitting superfluous zeros in the upper half of matrix $A\in V$. For instance
\begin{center}
\begin{picture}(150, 50)
\put(0,20){$e_{_A}=$}
\put(30,0){\line(1,0){50}}
\put(30,0){\line(0,1){50}}
\put(30,50){\line(1,-1){50}}
\put(32,33){0}
\put(32,18){$\alpha$}
\put(32,3){$\beta$}
\put(47,18){$0$}
\put(47,3){$\gamma$}
\put(62,3){$0$}
\put(82,0){,}
\put(100,20) {$\alpha, \beta, \gamma\in \F_q$}
\end{picture}
\end{center}
denotes the idempotent $e_{_A}=[A]\in KV$ affording the linear character $\chi_{_A}\in \hat V$ with $$A=
\begin{pmatrix}
 0 & 0 & 0\\
\alpha & 0&0 \\
 \beta & \gamma &0 
\end{pmatrix}\in V.$$
For $A\in V, u\in U$ we denote the matrix $\overline{Au^{-t}}\in V$ by $A\ldotp u$. This defines indeed a permutation action of $u$ on $V$. Moreover, using theorem \ref{jed} we derive a monomial action of $u\in U$ on the idempotent  basis $\{[A]\,|\,A\in V\}$ of $KV$ interpreting again linear characters of $V$ as elements of the group $KV$:
\begin{Cor}\label{2.7}
 Let $A\in V$ and $u\in U$. Then 
$$[A] u=\chi_{_{A\ldotp u}}(u-1)[A\ldotp u].$$
\end{Cor}

For $1\leqslant j < i \leqslant n$ and $\alpha\in \F_q$ recall that $x_{ij}(\alpha)=1+\alpha \epsilon_{ij}\in U$ and that the root subgroup $X_{ij}=\{x_{ij}(\alpha)\,|\,\alpha\in \F_q\}$ is an abelian subgroup of $U$ isomorphic to $(\F_q,+)$. Moreover $x_{ij}(\alpha)$ acts on any matrix $A\in M_{n\times n}(\F_q)$ by the elementary column operation adding $\alpha$ times column $i$ to column $j$ in $A$. Now $x_{ij}(\alpha)^{-t}=x_{ji}(-\alpha)$, hence $A\ldotp x_{ij}(\alpha)$ is obtained from $A$ by adding $-\alpha$ times column $j$ to column $i$ (from left to right) and setting nonzero entries in the resulting matrix at position on or to the right of the diagonal to zero. We call this maneuver ``truncated column operation'' (from left to right). Since every element of $U$ can be written uniquely as product $u=\prod_{1\leqslant j < i \leqslant n}x_{ij}(\alpha_{ij})$ for $\alpha_{ij}\in \F_q$, where the product is taken in an arbitrary but fixed linear order of the indices $(i,j)$, the (permutation) action of $u\in U$ on $[A]$ for $A\in V$ can be described by the corresponding sequence of truncated column operations. Moreover
\begin{equation}\label{2.8}
 \chi_{A\ldotp x_{ij}(\alpha)}(x_{ij}(\alpha)-1)= \chi_{_{A\ldotp x_{ij}(\alpha)}}(\alpha e_{ij})=\theta(\alpha A_{ij})
\end{equation}since column $j$ coincides in $A$ and $A.x_{ij}(\alpha)$.

\begin{Cor}\label{2.9}
 Let $1\leqslant j < i \leqslant n$, $\alpha\in \F_q$ and $A\in V$. Then 
$$[A]x_{ij}(\alpha)=\theta(\alpha A_{ij})[A\ldotp x_{ij}(\alpha)].$$
\end{Cor}
Similarly the left operation of $u$ on the idempotent basis $\{[A]\,|\,A\in V\}$ of $V$ can be described by sequences of truncated row operations from down up, the coefficient in $\C$ being obtained similarly.

\begin{Remark}\label{2.10}
 The basis $\{\chi_{_A}\,|\,A\in V\}$ of $\C \hat V$ mapped back into $\C U\cong \C^U$ by extending $f^{-1}: V \longrightarrow U:  A\mapsto A+1\in U$ be linearity is Yan's Fourier basis of $\C U$, dual to the basis $\{f^{-1}[A]\,|\,A\in V\}$ of $\C U$.
\end{Remark}

We set $\E=\{[A]\,|\,A\in V\}$. The orbits of $U$ acting on $\E$ can now be described combinatorially using \ref{2.9} and its left handed analogue.

\begin{Defn}\label{2.11}
 Let $1\leqslant j < i \leqslant n$. The arm $h_{ij}^a$ centred at $(i,j)$ consists of all positions $(i,k)\in \Phi^+$ strictly to the right of $(i,j)$, thus $h_{ij}^a=\{(i,k)\,|\,j<k<i\}$, and the hook leg $h_{ij}^l$ is the set of positions $(l,j)\in \Phi^+$ strictly above $(i,j)$, thus $h_{ij}^l=\{(l,j)\,|\,j<l<i\}$. Finally the hook $h_{ij}$ centred at $(i,j)$ is defined to be $h_{ij}=h_{ij}^a\cup h_{ij}^b\cup \{(i,j)\}$. This may be illustrated as 
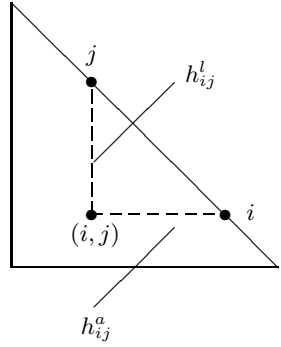
\begin{center}
 \begin{picture}(120,130)
  \put(0,25){\line(1,0){100}}
\put(0,25){\line(0,1){100}}
\put(0,125){\line(1,-1){100}}
\put(30,45){\circle*{4}}
\put(80,45){\circle*{4}}
\put(30,95){\circle*{4}}
\multiput(30,45)(7,0){7}{\line(1,0){4}}
\multiput(30,45)(0,7){7}{\line(0,1){5}}
\put(88,43){\footnotesize$i$}
\put(28,103){\footnotesize$j$}
\put(22,35){\footnotesize$(i,j)$}
\put(31,65){\line(1,1){30}}
\put(65,95){\footnotesize$h_{ij}^l$}
\put(61,40){\line(-1,-1){30}}
\put(26,0){\footnotesize$h_{ij}^a$}
 \end{picture}
\end{center}
\end{Defn}

Let $[A]\in \E$, and let $\tilde \O$ be the $U$-$U$-biorbit containing $[A]$. If $A$ is the zero matrix, $\O=\{[A]\}$. Otherwise let $A_{ij}\not=0$ be the lowest non zero entry in  the first non zero column $j$ of $A$ from the left. Acting by truncated row- and column operations, i.e. by elements of $U$ of the form $x_{ik}(\alpha_k)$ from the left and $x_{lj}(\beta_l)$ from the right, for suitable $\alpha_k, \beta_{l}\in \F_q$,  $l, k=j+1,\ldots,i-1$ we find $[B]\in \O$ such that $B_{ij}=A_{ij}$ and $B_{kl}=0$ for all $(k,l)\in h_{ij}^a\cup h_{ij}^l$. Choosing $0\not=B_{st}$ to be the lowest non zero value in $B$ in the first non zero column $t$ strictly to the right of column $j$ (so $j<t<s\leqslant n-1$) and proceeding in the same way, we obtain an idempotent $[C]\in \O$ with $A_{ij}=B_{ij}=C_{ij}, C_{st}=B_{st}$, such that all entries on the hook arms $h_{ij}^a$, $h_{st}^a$ of $C$ are zeros. Proceeding in this way we finally find an idempotent $[D]$ in $\O$ such that in $D$ in each  row and in each column there is at most one non zero entry. Such idempotents are called {\bf verge idempotents} and we have shown that  each $U$-$U$-biorbit $\O$ of $\E$ contains a verge idempotent. One sees easily that each biorbit contains at most one verge idempotent and we have shown that there is a bijection between verges and $U$-$U$-biorbits on $\E$. Obviously the idempotent  $[A]\in \E$ for the  zero matrix $A=0$ affords the trivial representation of $V$.

\begin{Defn}\label{2.12}
 A (right) {\bf template} is an idempotent $[A]\in \C V$, $A\in V$ such that the entries to the right of the lowest non zero entries in each column of $A$ are zero. So if $A_{ij}\not=0$, $1\leqslant j<i \leqslant n$, but $A_{kj}=0$ for $i+1\leqslant k \leqslant n$, then all entries on $h_{ij}^a$ are zero.

\medskip

A {\bf main condition} of a template $[A]$ is a position $(i,j)$ such that $A_{ij}$ is the lowest non zero entry in column $j$ of $A$. Note that the main conditions of templates are in different rows by construction. We denote the set of main conditions of the template $[A]$ by $\main[A]$ and call $[B]\in \C V$ with $ B=\sum_{(i,j)\in \main(A)} A_{ij}e_{ij}$, the {\bf verge} of the template $[A]$, denoted by $\verge[A]=[B]$. Note that by construction $\verge[A]$, $[A]$ a template, is indeed a {\bf verge idempotent}, that is $\verge[A]$ has in each row and each column at most one non zero entry. 
We may illustrate this as follows:
\end{Defn}
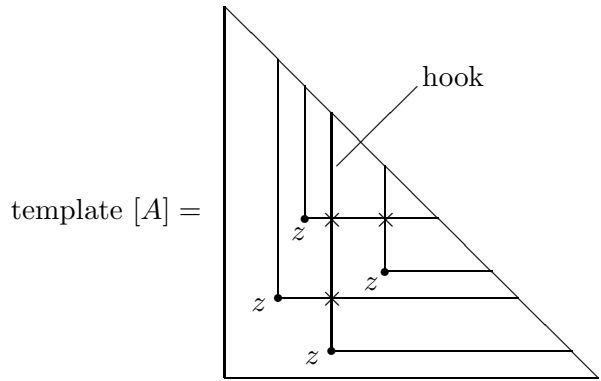
\begin{center}
 \begin{picture}(190,145)
\put(80,0) {\line(1,0){140}}
\put(80,0){\line(0,1){140}}  
\put(80,140){\line(1,-1){140}}
\put(0,60){template $[A]=$}
\put(100,30) {\line(1,0){90}}
\put(100,30){\line(0,1){90}}
\put(140,40) {\line(1,0){40}}
\put(140,40){\line(0,1){40}}  
\put(110,60) {\line(1,0){50}}
\put(110,60){\line(0,1){50}}
\put(120,10) {\line(1,0){90}}
\put(120,10){\line(0,1){90}}
\put(100,30){\circle*{3}}
\put(140,40){\circle*{3}}
\put(110,60){\circle*{3}}
\put(120,10){\circle*{3}}
\put(90,25){$z$}
\put(132,33){$z$}
\put(105,52){$z$}
\put(110,5){$z$}
\put(116,27){$\times$}
\put(116,57){$\times$}
\put(136,57){$\times$}
\put(122,80) {\line(1,1){30}}
\put(154,110) {hook}
 \end{picture}
\end{center}
The position denoted by $z$ are main conditions and have non zero entries in $A$ . All other possible non zero entries are located in columns with main conditions strictly above those and not on hook intersection, that is to the right of a main condition, indicated by $\times$. The positions strictly above main conditions not to the right of main conditions are called {\bf supplementary conditions} and their set is denoted by $\suppl(\p)$, where $\p=\main[A]$. Thus $A_{st}\not=0$ implies $(s,t)\in \p$ or $(s,t)\in \suppl(\p)$.

\begin{Theorem}[Yan]\label{2.13}
 Each right $U$-orbit $\O$ contains precisely one template $[A]$. Moreover if $\O$, $\O'$ are right $U$-orbits on $\E$ with templates $[A]$ respectively $[A'], A, A'\in V$, then $\C \O\cong \C \O'$ if and only if $\verge[A]=\verge[A']$. If $\verge[A]\not=\verge[A']$, then $\C \O$ and $\C \O'$ afford orthogonal characters, that is have no irreducible constituent in common.
\end{Theorem}

In particular $[A] \C U\cong \verge[A] \C U$ for any template $[A]$ in $\E$. An isomorphism is obtained by a sequence of truncated row operations applied to $\verge[A]$ yielding the template $[A]$. The characters afforded by the right $U$-orbits $\O$ of $\E$ are supercharacters in Yan's supercharacter theory for $U$. By \ref{2.13}  for any idempotent $[B]\in \E$, we can find a unique verge idempotent $[A]\in \E$ such that $[B]\C U\cong [A]\C U$, so it suffices to investigate the supercharacters arising from orbits generated by verge idempotents and we can extend the definition \ref{2.12} by denoting $\main[B]=\main [A]$ and $\verge[B]=\verge[A]$.

\section{Hom-spaces and irreducibles of minimal dimension}
Most of the material presented in this section is known or can easily be derived from the existing literature (e.g. \cite{super} or \cite{yan}). However, since we employ a representation theoretic approach using Hom-spaces between orbit modules explicitely, we provide proofs as well.

\medskip

For $[A]\in \E, A\in V$, we denote the right- respectively left $U$-orbit in $\E$ containing $[A]$ by $\O_A^r$ and $\O_A^l$ respectively. Recall that $\C \E$ is isomorphic to $\C U$ as $U$-$U$-bimodule (see Remark \ref{2.10}), an isomorphism given by the inverse $f^{-1}$ of the (bijective) 1-cocycle $f: U\rightarrow V.$ More precisely
\[
f^{-1}[A]=\frac{1}{|V|}\sum_{B\in V} \overline{\chi_{_A}(B)}(B+1)\in \C U
\]
extends by linearity to an $U$-$U$-bimodule isomorphism. In particular $\C \O^r_A$ and $\C \O^l_A$ are isomorphic under $f^{-1}$ to the right respectively the left ideal of $\C U$ generated by $f^{-1}[A]$.

\begin{Prop}\label{3.1}
Let $[A], [B]\in\E$. Then $\Hom_{\C U}(\C \O^r_A, \C\O^r_B)$ has $\C$-basis $f^{-1}(\O^l_A\cap \O_B^r)$, where $x\in f^{-1}(\O^l_A\cap \O_B^r)$ acts on $\C \O^r_A$ by left multiplication.
\end{Prop}
\begin{proof}
$\C U$ is a self injective algebra hence, for any $x, y\in \C U$ every homomorphism from $x\C U$ to $y \C U$ is obtained by left multiplication by some element of $\C U$. As an easy consequence we have
\begin{eqnarray*}
 \Hom_{\C U} ([A] \C U, [B] \C U) \cong f^{-1}(\C U[A]\cap [B] \C U)\cong \C (\O^l_A\cap \O^r_B)\end{eqnarray*} as $\C$-vector space, since $\O^l_A, \O^r_B \subseteq \E$. We conclude that $f^{-1}(\O^l_A\cap \O^r_B)$ gives a $\C$-basis of $\Hom_{\C U}(\C \O^r_A, \C \O_B^r)$.
\end{proof}

\begin{Defn}\label{3.2}
 Let $M, N$ be $\C U$-modules. We say that $M$ and $N$ are {\bf disjoint}, if they have no irreducible constituents in common, that is if the characters afforded by $M$ and $N$ are orthogonal.
\end{Defn}

Proposition \ref{3.1} enables us to give a proof of part of Yan's theorem \ref{2.13}.

\begin{Cor}\label{3.3}
 Let $\O_1, \O_2\subseteq \E$ be right $U$-orbits. Then either $\C \O_1\cong \C \O_2$ or $\C O_1$ and $\C O_2$ are disjoint. 
\end{Cor}
\begin{proof}
 Let $\Hom_{\C U}(\C \O_1, \C \O_2)\not=(0)$. Let $[A]\in \O_1$ and let $\tilde \O_1$ be the left $U$-orbit in $\E$ generated by $[A]$. By \ref{3.1} we find $[B]\in \tilde \O_1\cap \O_2$, thus there exists $u\in U$ such that $u[A]=\kappa[B]$ with $[B]\in \O_2, \kappa\in \C^*$. Thus $u [A] \C U=[B]\C U=\C \O_2$ and left multiplication by $u$ maps $\C \O_1$ surjectively onto $\C \O_2$. But left multiplication by $u$ is invertible and therefore $\C \O_1\cong \C \O_2$, as desired. 
\end{proof}

\begin{Defn}\label{3.4}
 Let $[A]\in \E$. The {\bf projective stabilizer} $\Pstab_u[A]$ of $[A]$ in $U$ is defined to be
\[
\Pstab_U[A]=\{u\in U\,|\, [A] u=\lambda_u [A], \, \lambda_u \in \C ^*\}.
\]
Thus $\Pstab_U[A]$ acts on $\C [A]$ by the linear character $\tau: u\mapsto \lambda_u\in \C^*$. By general theory we have
\[
[A]\C U\cong \Ind^U_{\Pstab_U [A]} \C_\tau,
\]
where $\C_\tau$ denotes the one dimensional $\C \Pstab_u[A]$-module affording $\tau$. The left projective stabilizer $\Pstab_U^{\ell}[A]$ is defined analogously.
\end{Defn}

We illustrate the action of root subgroups $X_{kl}$ ($1 \leqslant l <k \leqslant n$) on idempotents as follows
\begin{equation}
 \begin{picture}(100,100)
\put(0, 40){$[A]=$}
 \put(40,0) {\line(1,0){90}}
 \put(40,0) {\line(0,1){90}}
 \put(40,90) {\line(1,-1){90}}
\put(60,20) {\line(1,0){50}}
 \put(60,20) {\line(0,1){50}}
\put(60,20){\circle*{2}}
\put(50,12){\tiny$(i,j)$}
\put(75,12){\tiny$(i,l)$}
\put(80,20){\circle*{2}}
\put(60,50){\circle*{2}}
\put(115,18){\footnotesize$i$}
\multiput(80,20)(0,6.2){5}{\line(0,1){5}}
\multiput(60,50)(5,0){4}{\line(1,0){4.5}}
\put(84,49){\footnotesize$l$}
\put(62,72){\footnotesize$j$}
\put(43,49){\tiny$(l,j)$}
\end{picture}      
\end{equation}
Suppose $[A]\in \E, A_{ij}\not=0$ ($1\leqslant j< i \leqslant n$). Let $j<l<i$. Then left operation on $[A]$ by $x_{il}(\alpha), \alpha\in \F_q$ adds $-\alpha$ times row $i$ to row $l$, hence changes entry $(l,j)$ of $A$ to $A_{lj}-\alpha A_{ij}$, leaving all entries not below the diagonal unchanged to be zero by truncation. Similarly right operation on $[A]$ by $x_{lj}(\beta), \beta\in \F_q$ adds $-\beta$ times column $j$ to column $i$ and truncating the resulting matrix, and hence changing entry $A_{il}$ to $A_{il}-\beta A_{ij}$. By \ref{2.9} the resulting idempotents have to be multiplied by the scalar $\theta(\alpha A_{il})$ in the first and $\theta(\beta A_{lj})$ in the second case to complete the monomial action of \ref{2.9}. Using this we obtain the following result immediately.

\begin{Theorem}\label{3.6}
 Let $[A]\in \E$ be a right template with $\main[A]=\p_{_A}=\{(i_1,j_1),\ldots,(i_k,j_k)\}\subseteq \Phi^+$. Then $\Pstab_U[A]$ is a pattern subgroup $U_{\cal R}$ with 
\[
{\cal R}=\{(r,s)\in \Phi^+\,|\,s\notin \{j_1,\ldots,j_k\}\}\cup\{(r,j_\nu)\,|\,\nu=1,\ldots,k, i_\nu\leqslant r \leqslant n\}.
\] 
Thus $\p_{_A}\subseteq {\cal R}$ and ${\cal R}^\circ ={\cal R}\setminus \p_{_A}$ is closed. Moreover $U_{{\cal R}^\circ }$ acts trivially on $[A]$, $U_{{\cal R}^\circ }\trianglelefteq U_R$ and $U_{\cal R}/U_{{\cal R}^\circ }\cong X_{i_1 j_1}\times \cdots\times X_{i_k j_k}$
acting on $[A]$ by the linear character $\theta_A=\theta_1\times \theta_2\times \cdots \theta_k$, where
$\theta_{\nu}: X_{i_\nu j_\nu}\rightarrow \C^*$ sends $x_{i_\nu j_\nu}(\alpha)$ to $ \theta(A_{i_\nu j_\nu}\cdot \alpha)\in \C^*$ for $\alpha\in \F_q,\, \nu=1,\ldots,k$.
\end{Theorem}

Thus ${\cal R}$ consists of all positions in $\Phi^+$ in zero columns of $A$ together with all positions on and below the positions in $\p_{_A}$.

\smallskip

Analogously $[A]\in \E$ is a {\bf left template}, if the nonzero entries of $A$ are only in rows, containing a main condition, and are besides the main conditions itself to the right of main conditions and not on hook intersections, that is below  main conditions. Define
\[
{\cal L} =\{(i,j)\in \Phi^+\, | \, i\notin\{i_1,\ldots,i_k\}\}\cup \{(i_\nu,l)\,|\, \nu=1,\ldots,k \text{ and } 1\leqslant l\leqslant j_\nu\}. 
\]
Thus $\cal L$ consists of all positions in zero rows of $[A]$ together with all position to the left of main conditions including main conditions. Set ${\cal L} ^\circ={\cal L}\setminus \p_{_A}$, then $U_{\cal L}=\Pstab_U^l [A]$, $U_{{\cal L} ^\circ}$ is normal in $U_{\cal L}$ and acts trivially on $[A]$ from the left and $U_{\cal L}/U_{{\cal L} ^\circ}=U_{\cal R}/U_{{\cal R} ^\circ}=X_{i_1 j_1}\times \cdots\times X_{i_k j_k}$ acting on $[A]$ by the left with the linear character $\theta_A$ again.

\bigskip

We describe now orbit modules in more detail. In view of theorem \ref{2.13} we may concentrate on verge orbits, that is orbits generated by some verge idempotent $[A]\in \E$. Let $\p=\p_{_A}=\main[A]$. Thus $A=\sum_{(i,j)\in \p} A_{ij}\epsilon_{ij}$. Let $\p=\{(i_1,j_1),\ldots,(i_k,j_k)\}\subseteq \Phi^+$.  Denote the hook centered at $(i_\nu, j_\nu) $ by $h_\nu, \nu=1,\ldots,k $. Set $a_\nu=i_\nu-j_\nu-1, a=a_1+\cdots+a_k$, and let $b$ be the number of hook intersections of the hooks $h_\nu$, $\nu=1,\ldots,k$.  Let $\T_A$ be the set of templates $[C]$ with $\verge[C]=[A]$. Thus $[C]\in \T_A$ if and only if $C_{i_\nu j_\nu}=A_{i_\nu j_\nu}, \nu=1,\ldots,k$ and the only other nonzero entries in $C$ are on the hook legs $h_\nu ^l$ different from the hook intersections.
We illustrate this in the following example:

\begin{Example}\label{3.7}
Suppose $[A]\in \E$ is a verge idempotent and let $\p=\{(i,j), (r,s)\}=\main[A]$ with $1\leqslant j<s<i<r\leqslant n$. Then $h_{ij}\cap h_{rs}=(i,s)$:
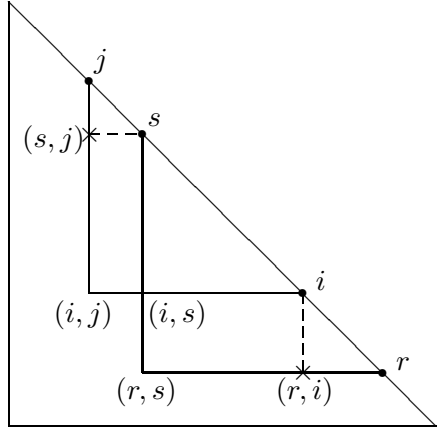
\begin{center}
 \begin{picture}(140,160)
  \put(0,0){\line(1,0){160}}
\put(0,0){\line(0,1){160}}
\put(0,160){\line(1,-1){160}}
\put(30,50){\line(1,0){80}}
\put(30,50){\line(0,1){80}}
\put(50,20){\line(0,1){90}}
\put(50,20){\line(1,0){90}}
\put(17,39){$(i,j)$}
\put(52,39){$(i,s)$}
\put(40,10){$(r,s)$}
\put(30,130){\circle*{3}}
\put(32,135){$j$}
\put(50,110){\circle*{3}}
\put(52,113){$s$}
\put(110,50){\circle*{3}}
\put(115,51){$i$}
\put(140,20){\circle*{3}}
\put(145,21){$r$}
\multiput(110,50)(0,-6.4){5}{\line(0,-1){4}}
\put(100,10){$(r,i)$}
\put(106,17){$\times$}
\multiput(30,110)(6.8,0){3}{\line(1,0){4}}
\put(5,105){$(s,j)$}
\put(26,107.2){$\times$}
 \end{picture}
\end{center}
Then $ [C]\in \T_A$ if and only if $A_{ij}=C_{ij}, A_ {rs}=C_{rs}$ and $C_{is}=0$, where all other nonzero entries of $C$ are on positions in columns $j$ and $s$ above $(i,j)$ and $(r,s)$ respectively. Inspecting the action of root subgroups $X_{rt}$ on $[A]$ from the left, we observe that for $\alpha\in \F_q$ using the left hand analog of corollary \ref{2.7}:
\[x_{rt}(\alpha) [A]=\theta(\alpha A_{rt})[B],\]
where  $B\in V$ is obtained from $A$ by adding $-\alpha$ times row $r$ to row $t$ in $A$ and setting in the resulting matrix all entries above and on the main diagonal back to zero. Since the only non zero entry in row $r$ of $A$ is $A_{rs}$, we have $\theta(\alpha A_{rt})=\theta(0)=1$, for $t\neq s$. Moreover, if $t\leqslant s$, we obtain by truncation $[B]=[A]$. For $s<t<r$, $B$ differs from $A$ only at position $(t,s)$, indeed $B_{ts}=-\alpha A_{rs}\neq 0$, since $A_{ts}=0$. Sequences of truncated row operations by the action of root subgroups $X_{rt}, 1\leqslant t<r$ on row $r$ will just place arbitrary entries on the hook leg $h_{rs}^l$ leaving other entries of $A$ unchanged. Similarly left action by sequences of elements of root subgroups $X_{it}$ fills the positions of $h_{ij}^\ell$ with arbitrary values of $\F_q$, leaving other entries of $A$ unchanged. All other root subgroups $X_{ab}$ not in row $r$ or row $i$ (i.e. $a\neq r, a\neq i$) belong to the projective left stabilizer of  $[A]$ in $U$. We conclude that $|\O^l_A|=q^a$ with $a=(i-j-1)+(r-s-1)=|h_{ij}^\ell|+|h_{rs}^\ell|$. Similarly for the right action of $U$ on $[A]$ we see that $X_{ab}\in \Pstab_U [A]$ for $s\neq b\neq j$. The right action of $X_{ts}, s<t<r$ fills position $(r,t)$ on $h_{rs}^a$ with arbitrary chosen values of $\F_q$. Replacing $(r,s)$ by $(i,j)$ the root subgroups $X_{tj}$, $j<t<i$ fill the positions on $h_{ij}^a$ in row $i$ of $A$ with arbitrary elements of $\F_q$. We conclude again, since $|h_{ij}^a|=|h_{ij}^\ell|=i-j-1$, $|h_{rs}^a|=|h_{rs}^\ell|=r-s-1,$ that $|\O^r_A|=q^a$.

Finally acting by $X_{ri}$ from the left as well as acting by $X_{sj}$ from the right on $[A]$ will fill position $(i,s)$ and leave all other entries of $A$ unchanged. Thus $\O_A^l\cap \O_A^r$ is precisely the set of idempotents $[B]\in \E$ with $A_{rs}=B_{rs}, A_{ij}=B_{ij}, B_{is}\in \F_q$ all other entries of $B$ being zero. We conclude $|\O_A^l\cap \O_A^r|=q=q^1$, that is $b=1$ and hence $|\T_A|=q^{a-b}=q^{a-1}$.
\end{Example}

\begin{Theorem}\label{3.8}
Let the verge $[A]\in \E$ be a verge with main conditions $\p=\p_{_A}$. Keeping the notation introduced above we have
\begin{enumerate}
\item [1)] $\O_A^r$ consists of all idempotent arising by filling all positions in $h_\nu^a$ ($\nu=1, \ldots, k$) with arbitrary elements of $\F_q$. Similarly $\O_A^l$ consists of all idempotents $[B]\in \E$, where $B$ coincides at main conditions with $A$ and where the hook legs $h_\nu^\ell$ are filled by arbitrary elements from $\F_q$.
\item [2)] $|\O_A^l|=|\O_A^r|=q^a$.
\item [3)] $|\T_A|=q^{a-b}$.
\item [4)] For every $[B]\in \T_A$ there exists $u\in U$ such that $u[A]=[B]$. In particular $\C \O^r_B\cong \C \O_A^r$.
\end{enumerate}
\end{Theorem}
\begin{proof}
Everything follows at once by inspecting example \ref{3.7}.
\end{proof}

Note that part 4) of \ref{3.8} reproves part of Yan's theorem \ref{2.13}. Moreover, using \ref{3.1}, one easily derives a complete proof of \ref{2.13}, proving that $\C \O^r_A\cong \C \O_B^r$ for a template $[B]\in \E$ if and only if $[B]\in \T_A$. Of course there is an analogous statement, using left templates, for left orbits. Since $\C \E$ is isomorphic to $\C U$ as $U$-$U$-bimodule we conclude, that the biorbit module $\C U[A] \C U$ generated by a verge $[A]\in \E$ is the direct sum of Wedderburn components of $\C U$. In particular, if $S$ is an irreducible submodule of $\C \O_A^r$, $S$ occurs in $\C U_{\C U}$ with multiplicity at least $|\T_A|=q^{a-b}$, and hence $\dim_\C S\geqslant q^{a-b}$, where $a$ and $b$ are defined as above. We have shown:

\begin{Theorem}\label{3.10}
Let $[A]\in \E$ be a verge, and let  $S$ be an irreducible constituent of $\C \O_A^r$. Then $\dim_\C S \geqslant q ^{a-b}$. If $\dim_\C S=q^{a-b+m}$ for some $m\in \N$, then the multiplicity of $S$ as constituent of $\C\O_A^r$ is given as $q^m$.
\end{Theorem}
\begin{proof}
We already proved that $q^{a-b}$ is a lower bound for the dimension of $S$. Let $\dim_\C S=q^{a-b+m}$ and let $s$ be the multiplicity of $S$ as irreducible constituent in $\C \O_A^r$. Then $s$ is the multiplicity of $S$ in $\C \O_B^r$ for all $[B]\in \T_A$ and so  $s\cdot q^{a-b}$ is the multiplicity of $S$ as constituent in the regular $U$-module $\C U_{\C U}$, and hence its dimension as well. So $s\cdot q^{a-b}=q^{a-b+m}$, proving $s=q^m$.
\end{proof}

In the setting of theorem \ref{3.10} we obtain as well an upper bound for the dimension of $S$ observing that $q^m \dim_\C S$ has to be less or equal the size $q^a$ of $\O^r_A$. So $q^m q^{a-b+m}\leqslant q^a$, that is $2m\leqslant b$.

\begin{Cor}\label{3.10n}
Let $[A] \in \E$ be a verge and let $S$ be an irreducible constituent of $\c \O_A^r$. Let $a,b$ be defined as above. Then
\[q^{a-b}\leqslant \dim_\C S \leqslant q^{a+c}\]
where $c$ is the largest integer less or equal $-\frac{b}{2}$. 
\end{Cor}

We say $S\leqslant \C \O_A^r$ has   {\bf minimal dimension} if and only if $\dim_\C (S)=q^{a-b}$. By standard arguments on endomorphism rings of semisimple modules we conclude:
\begin{Cor}\label{3.11}
Let $[A]\in \E$ be a verge, $S\leqslant \C \O_A^r$ be irreducible and let  $\epsilon_{_S}: \C \O_A^r \rightarrow S \hookrightarrow \C \O_A^r $ be the canonical projection. Thus   $\epsilon_{_S}\in E=\End_{\C U}(\C \O_A^r)$. Then the following  holds:
\begin{enumerate}
\item [1)] $S=\C \O_A^r$, that is $\C \O_A^r$ is irreducible  if and only if   the hooks centered at the main conditions do not intersect, (Yan).
\item  [2)] If $\dim_\C S=q^{a-b+m}$, $a, b,m\in \N$ as above, then $\epsilon_{_S} E$ is an irreducible $E$-module of dimension $q^m$.
\item  [3)] $S$ has minimal dimension if and only if $\epsilon_{_S}E=\C \epsilon_{_S}$ is a  one dimensional $E$-module.
\end{enumerate}
\end{Cor}


\section{Endomorphism rings}\label{4}

Throughout this section let $[A]\in\E$ be a verge with main conditions $\p=\p_{_A}=\{(i_1,j_1) \ldots,(i_k,j_k)\}$ $\subseteq \Phi^+$. Recall that $\cR^\circ \subseteq \Phi^+$ denotes the set of all positions in $\Phi^+$ in zero columns in $[A]$, together with all positions strictly below the main conditions in columns $j_1, \ldots,j_k$ in $[A]$ and set again $\cR=\cR^\circ\cup \p$. Similarly we have $\cL^\circ\subseteq \cL \subseteq \Phi^+$ for the left action of $U$, (compare \ref{3.6}). We define
\begin{equation}\label{4.1}
\begin{matrix}
Y^r&=&\{u\in U\,|\, u[A]\in \C \O_A^r\}\\
Y^l&=&\{u\in U\,|\, [A]u\in \C \O_A^l\}
\end{matrix}
\end{equation}
Obviously $Y^r, Y^l$ are subgroups of $U$ and by \ref{3.1} $E^r=\End_{\C U} \C \O_A^r, E^l=\End_{\C U} \C \O_A^l$ are epimorphic images of $\C Y^r$ and $\C Y^l$ respectively. More precisely, since every endomorphism $h\in E^r$ is completely determined by its action on the generator  $[A]$ of $\c \O_A^r$, we have
\begin{Lemma}\label{4.2}
The left annihilator $\mathfrak a_r=\{x\in \C Y^r\,|\, x[A]=0\}$ is an ideal in $\C Y^r$ such that $E^r=\C Y^r/\mathfrak a_r$. Similarly $E^l=\C Y^l/\mathfrak a_l$ for the right annihilator $\mathfrak a_l$ of $[A]$ in $\C Y^l$.
\end{Lemma}

Note that the indices ``$r$'' and ``$l$'' indicate here which action is centralized, even if the algebras $E^r$ and $E^l$ act from the left and the right respectively.

In \ref{3.1} we exhibited a $\C$-basis of $E^r, E^l$ to be labeled by $\O_A^l\cap \O_A^r$. More precisely, for any $[B]\in \O_A^l\cap \O_A^r$ we may choose $u_{_B}\in U$ such that $u_{_B}[A]\in \C [B]$. Then $\{u_{_B}\,|\,[B]\in \O_A^l\cap \O_A^r\}\subseteq Y^r$ is modulo $\frak a_r$ a basis of $E^r$.

\begin{Lemma}\label{4.3}
$[B]\in \O_A^l\cap \O_A^r$ if and only if $B$ coincides with $A$ at all positions except the hook intersections.
\end{Lemma}
\begin{proof}
By \ref{3.8} $\O_A^l$ consists of all idempotents $[B]\in \E$ which differ from $[A]$ only in positions on the hook legs $h_{ij}^\ell$ with $(i,j)\in \p$ and $\O_A^r$ consists of idempotents $[C]\in \E$ differing from $[A]$ only on the hook arms $h_{ij}^a$ with $(i,j)\in \p$. This implies the lemma immediately.
\end{proof}

Let $(i,j), (r,s)\in \p$, $1\leqslant j<s <i<r \leqslant n$. Thus $h_{ij}, h_{rs}$ intersect precisely in position $(i,s)$ (compare \ref{3.7}). Moreover acting by $X_{ri}$ from the left on $[A]$ produces all idempotents $[B]\in \O_A^l\cap \O_A^r$, which differ from $[A]$ only in position $(i,s)$, the hook intersection. Note that $(r,i)\in h_{rs}^a$. For dealing with more than one  hook intersection and keeping track of all the positions involved, we inspect the following example:

\begin{Example}\label{4.4}
Let $1\leqslant j<s<b<i<r<a\leqslant n$ with $\{(i,j), (r,s), (a,b)\}\subseteq \p_{_A}$.
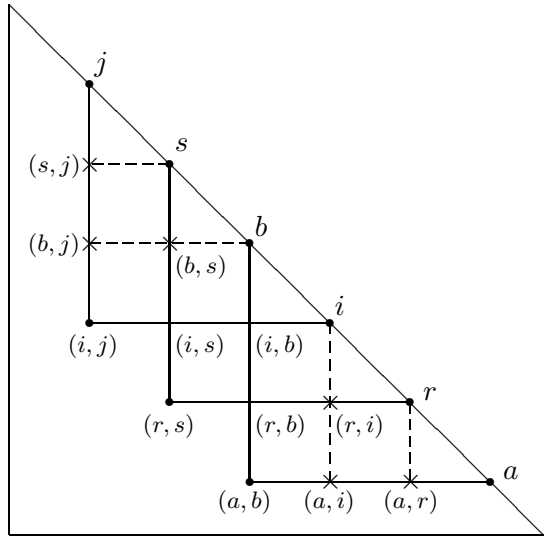
\begin{center}
\begin{picture}(130,200)
\put(-40,0){\line(1,0){200}}
\put(-40,0){\line(0,1){200}}
\put(-10,80){\line(1,0){90}}
\put(-10,80){\line(0,1){90}}
\put(-40,200){\line(1,-1){200}}
\put(20,50){\line(1,0){90}}
\put(20,50){\line(0,1){90}}
\put(50,20){\line(0,1){90}}
\put(50,20){\line(1,0){90}}
\put(10,39){\footnotesize $(r,s)$}
\put(52,39){\footnotesize $(r,b)$}
\put(82,39){\footnotesize $(r,i)$}
\put(38,10){\footnotesize $(a,b)$}
\put(20,140){\circle*{3}}
\put(20,50){\circle*{3}}
\put(-10,80){\circle*{3}}
\put(50,20){\circle*{3}}
\put(22,145){$s$}
\put(-10,170){\circle*{3}}
\put(-8,175){$j$}
\put(50,110){\circle*{3}}
\put(52,113){$b$}
\put(80,80){\circle*{3}}
\put(82,83){$i$}
\put(110,50){\circle*{3}}
\put(115,51){$r$}
\put(140,20){\circle*{3}}
\put(145,21){$a$}
\multiput(110,50)(0,-6.4){5}{\line(0,-1){4}}
\multiput(80,80)(0,-7){9}{\line(0,-1){4}}
\put(100,10){\footnotesize $(a,r)$}
\put(70,10){\footnotesize $(a,i)$}
\put(106.1,17.2){$\times$}
\put(76.1,17.2){$\times$}
\put(76.1,47.2){$\times$}
\put(-18,69){\footnotesize $(i,j)$}
\put(22,69){\footnotesize $(i,s)$}
\put(52,69){\footnotesize $(i,b)$}
\multiput(50,110)(-6.2,0){10}{\line(-1,0){4}}
\multiput(20,140)(-6.2,0){5}{\line(-1,0){4}}
\put(22,99){\footnotesize $(b,s)$}
\put(-33,107){\footnotesize $(b,j)$}
\put(-33,137){\footnotesize $(s,j)$}
\put(16,107){$\times$}
\put(-14,107){$\times$}
\put(-14,137){$\times$}
 \end{picture}
\end{center}
We list the positions of hook intersections and the root subgroups changing the values there by left and right action in the following table:
\begin{center}
\begin{tabular}{|c|c|c|c|}\hline
\begin{tabular}{c}hook intersections\\ to be changed \end{tabular} & \begin{tabular}{c}$h_{ij}\cap h_{rs}$\\ $=(i,s)$ \end{tabular} &  \begin{tabular}{c} $h_{ij}\cap h_{ab}$\\ $=(i,b)$ \end{tabular} & \begin{tabular}{c}$h_{rs}\cap h_{ab}$ \\ $=(r,b)$\end{tabular}\\ \hline
\begin{tabular}{c}by left action =\\ truncated row operations\end{tabular}& $X_{ri}$ & $X_{ai}$ & $X_{ar}$\\ \hline
\begin{tabular}{c}by right action =\\ truncated column operations\end{tabular}& $X_{sj}$ & $X_{bj}$ & $X_{bs}$\\\hline
\end{tabular}
\end{center}
\end{Example}

So let $[B]\in \O_A^l\cap \O_A^r$ with $B$ differing from $A$ only in these three hook intersections we may act first by $X_{ri}$ to obtain $B_{is}$ at position $(i,s)$. Then left action by $X_{ai}$  and  $X_{ar}$ will change only position $(i,b)$ and $(r,b)$, thus we can insert $B_{ib}$ at position $(i,b)$ and $B_{rb}$ at position $(r,b)$ without changing the entry at position $(i,s)$. Similarly we use $X_{bs}$ on the right first, then $X_{sj}$ and $X_{bj}$ acting from the right on $[A]$ to obtain $[B]$.

It is easy to check $\cR \cup \{(s,j),(b,j),(b,s)\}$ and $\cL \cup \{(r,i),(a,i),(a,r)\}$ are closed subsets of $\Phi^+$.

\bigskip

In general, we define $\hat \cL$ to be $\cL$ combined with all positions $(r,i)\in \Phi^+$ such that there exist $1\leqslant j, s \leqslant n$ such that $(i,j), (r,s) \in \p_{_A}$ and $1\leqslant j<s<i<r\leqslant  n$. Then $\hat \cL$ consists of $\cL$ and all positions on hook arms, such that the corresponding root subgroups change in $[A]$ only the values at a hook intersection acting from the left. Thus obviously $|\hat \cL|-|\cL|=|\hat \cL \setminus \cL|=b=$number of hook intersections. We can now prove the following:

\begin{Theorem}\label{4.5}
$\hat \cL$ is a closed subset of $\Phi^+$ such that $U_{\hat \cL }=\{u\in U\,|\,u[A]\in \c \O_A^r\}=Y^R$. Moreover $U_{{\cL}^\circ}$ and $U_{\cal L}$ are normal subgroups of $U_{\hat \cL }$ and $U_{\cal L}/U_{{\cal L}^\circ }=X_{i_1 j_1}\times \cdots X_{i_k j_k}$ is contained in the center of    $U_{\hat \cL}/U_{{\cal L}^\circ }$ and acts on the verge idempotent $[A]$ from the left by the linear character $\theta_A$ defined in \ref{3.6}. If $\epsilon_{_A}\in \C U_{\cL}$ denotes the idempotent affording the linear character $\theta_A$ with trivial $U_{\cL^\circ}$-action,  the endomorphism ring $E^r=\End_{\C U} (\C \O_A^r)$ is isomorphic to  the right ideal  $\epsilon_{_A} \C U_{\hat \cL}$ of $\C U_{\hat \cL}$.
\end{Theorem}
\begin{proof}
First we show that  $(a,r), (r,i)\in \hat \cL$ implies $(a,i)\in \hat \cL$ proving that $\hat \cL$ is a closed subset of $\Phi^+$. Since $\cL$ is closed, we may assume that $(a,r)$ or $(r,i)$ is not contained in $\cL$.

Suppose $(a,r)\notin \cL$ thus we find $(a,b), (r,s)\in \p$, $1\leqslant s<b<r<a\leqslant n$ such that  $h_{ab}\cap h_{rs}=\{(r,b)\}$  and $X_{ar}$ acting from the left on $[A]$ changes only the entry at the hook intersection $(r,b)$ in $[A]$.

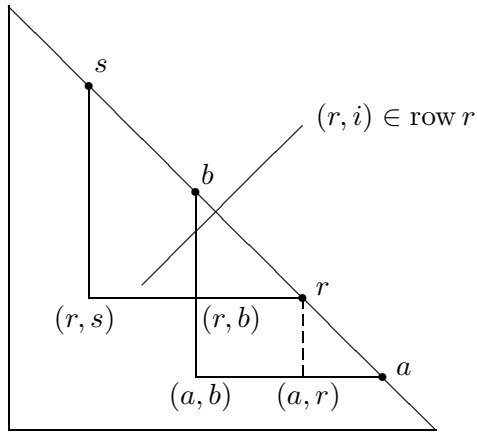
\begin{center}
 \begin{picture}(140,160)
  \put(0,0){\line(1,0){160}}
\put(0,0){\line(0,1){160}}
\put(0,160){\line(1,-1){160}}
\put(30,50){\line(1,0){80}}
\put(30,50){\line(0,1){80}}
\put(50,55){\line(1,1){60}}
\put(70,20){\line(0,1){70}}
\put(70,20){\line(1,0){70}}
\put(17,39){$(r,s)$}
\put(72,39){$(r,b)$}
\put(60,10){$(a,b)$}
\put(100,10){$(a,r)$}
\put(30,130){\circle*{3}}
\put(32,135){$s$}
\put(70,90){\circle*{3}}
\put(72,93){$b$}
\put(110,50){\circle*{3}}
\put(115,51){$r$}
\put(140,20){\circle*{3}}
\put(145,21){$a$}
\put(115,115){$(r,i)\in \row r$}
\multiput(110,50)(0,-6.4){5}{\line(0,-1){4}}
\end{picture}
\end{center}
Now $(r,i)\in \hat \cL$ is a position in row $r$, so $i<r$.

Suppose first  $(r,i)\in \cL$. Then $(r,i)$ is not to the right of $(r,s)\in \p\subseteq \cL$, that is $i\leqslant s$. Since $r<a$ and $s<b$ we have $i<b$ and  hence the position $(a,i)$ in row $a$ is strictly to the left of $(a,b)\in \p$. We conclude that $(a,i)\in \cL^\circ \subseteq \cL\subseteq \hat \cL$. Since the commutator subgroup $[X_{ar}, X_{ri}]$ equals $X_{ai}$ this shows too that $X_{ar}$ normalizes $U_{\cL^\circ}$ and $U_{\cL}$ and that it commutes with $X_{rs}$ modulo $U_{\cL^\circ}$.

Now suppose $(r,i)\notin \cL$ as well. So $(r,i)$ is in row $r$ strictly to the right of $(r,s)\in\p$, that is $s<i$. If $i\leqslant b$, then $(a,i)$ is in row $a$ not to right of $(a,b)\in \p$, hence is contained in $\cL$ and we are done, since $\cL\subseteq \hat \cL$. Suppose $b<i$. Since $(r,i)\notin \cL$, but $(r,i)\in \hat \cL$, there must be in addition a main condition $(i,j)$ in row $i$ with $j<s$ such that $X_{ri}$ acting from the left on $[A]$ changes only the value at the hook intersection $h_{ij}\cap h_{rs}=\{(i,s)\}$ in $[A]$. We have (compare \ref{4.4}):
$1\leqslant j<s<b<i<r<a\leqslant n$ combining all inequalities above, and $(i,j),(r,s),(a,b)\in \p$. In particular $h_{ij}\cap h_{ab}=\{(i,b)\}$. Moreover $X_{ai}$ acts on $[A]$ from the left changing in $[A]$ only the value $A_{ib}=0$ at position $(i,b)$, hence $(a,i)\in \hat \cL$, as desired.

A similar argument shows that $(r,i)\notin \cL$ and $(a,r)\in \cL$, then $(a,i)\in  \cL^\circ$ proving that $\hat \cL$ is closed, and that the commutator subgroup $[U_{\hat \cL}, U_\cL]$ is contained in $U_{{\cal L}^\circ}$. This implies $U_{{\cal L}^\circ}, U_{\cal L}$ are normal subgroups of $ U_{\hat \cL}$ and $U_{\cal L}/U_{{\cal L}^\circ }$ is contained in the center  $Z(U_{\hat \cL}/U_{{\cal L}^\circ})$ of $U_{\hat \cL}/U_{{\cal L}^\circ}$.

Obviously $U_{\hat \cL}\subseteq Y^r$ by \ref{3.1}. Since $U_{\cL^\circ}$ acts trivially and $U_{\cL}/U_{\cL^\circ}$ by the linear character $\theta_A$ from the left on $[A]$, $\dim_{\C} (\epsilon_{_A} \C U_{\hat \cL})=q^b=\dim_{\C} E^r$, finishing the proof of the theorem. 
\end{proof}

Obviously there is a left sided version of \ref{4.5}. Recall that $\cR^\circ$ (for given verge $[A]\in \E$) consists of all positions $(i,j)\in \Phi^+$ in zero columns and in column $j$ such that position  $(i,j)$ is strictly above some main condition $(r,j)\in \p_{_A}$, that is $r< i\leqslant n$. $\cR$ was defined to be $\cR^\circ \cup \p$. Now $\hat \cR$ is $\cR$ together with all positions $(a,b)\in \Phi^+$ satisfying the following: $(a,b)\in h_{ij}^\ell$ for some $(i,j)\in \p_{_A}$ (so $b=j$) and $X_{ab}$ acts on $[A]$ from the right by changing the value on precisely one hook intersection. 

\begin{Theorem}\label{4.5l}
$\hat \cR\subseteq \Phi^+$ is closed, $U_\cR, U_{\cR^\circ}\trianglelefteq U_{\hat \cR}$ and $U_{\cR}/U_{\cR^\circ}$ is central in $U_{\hat \cR}/U_{\cR^\circ}$. Right multiplication on $[A]$ by elements of $\C U_\cR$ induces endomorphism of the left orbit module $\C \O_A^l$. This endomorphism ring $E^l=\End_{\C U} \C \O_A^l$ is isomorphic to the left ideal $\C U\!_{\hat \cR}\epsilon_{_A}^\perp$ of the group algebra $\C U\!_{\hat \cR}$, where $\epsilon_{_A}^\perp$ is  the idempotent of $\C U_\cR$ affording the trivial $U_\cR^\circ$-action and affording the linear character $\theta_A$ of \ref{3.6} on $U_{\cR}/U_\cR^{\circ}$.
\end{Theorem}

\begin{Remark}\label{4.6}
Since $\epsilon_{_A}$ and $\epsilon_{_A}^\perp$ are central idempotents in $\C U_{\hat{\cL}}$ and $\C U_{\hat{\cR}}$ respectively we see that 
$\epsilon_{_A} \C U_{\hat \cL} =\C U_{\hat \cL} \epsilon_{_A}$ and $\epsilon_{_A}^\perp \C U_{\hat \cR} =\C U_{\hat \cR} \epsilon_{_A}^\perp$ are ideals of $\C U_{\hat \cL}$ and $\C U_{\hat \cR}$ respectively.
\end{Remark}

The next lemma shows, how we can shift the left action on $[A]$ of $U_{\hat \cL}$ to the right to obtain a right action by $U_{\hat \cR}$. Let $(i,j), (r,s)\in \p$ with $1\leqslant j<s<i<r\leqslant n$ as in example \ref{3.7}. Thus $h_{ij}\cap h_{rs}=\{(i,s)\}$ and $X_{ri}$ acting from the left and $X_{sj}$ acting from the right on $[A]$ will change only the value at the hook intersection $(i,s)$ in $[A]$. Note that $(r,i)\in \hat \cL\setminus \cL$ and $(s,j)\in \hat \cR\setminus \cR$. In this situation we have:

\begin{Lemma}\label{4.7}
Let $A_{rs}=\tau$, $A_{ij}=\rho$, so $0\not=\rho, \tau\in \F_q$. Set $\sigma=\tau/\rho\in \F_q$. Then
\[
x_{ri}(\alpha)[A]=[A]x_{sj}(\alpha \sigma) \text{ for all } \alpha\in \F_q.
\]
\end{Lemma}
\begin{proof}
We calculate, using the left hand sided version of \ref{2.9}:
\[
x_{ri}(\alpha)[A]=\theta(\alpha A_{ri})[B]=[B]\in \O_A^l \cap \O_A^r,
\]
using $A_{ri}=0$, where $B$ differs from $A$ only at position $(i,s)$ which is $B_{is}=-\alpha \tau\in \F_q$. Again by \ref{2.9}:
\[
[A]x_{sj}(\alpha\sigma)=\theta(\alpha\sigma A_{sj})[C]=[C], \text{ since } A_{sj}=0,
\]
where $C$ differs from $A$ only at position $(i,s)$ with $C_{is}=-\alpha \sigma\rho=-\alpha \cdot \frac{\tau}{\rho}\cdot \rho=-\alpha\tau=B_{is}$. So $B=C$ proving the lemma.
\end{proof}

For $(a,b)\in \Phi^+, \beta\in \F_q$, we define the linear character $\theta_{ab}^\beta: X_{ab}\rightarrow \C^*$ by $\theta_{ab}^\beta (x_{ab}(\alpha))=\theta(\alpha \beta)$. Observe that $\{\theta_{ab}^\beta\,|\, \beta\in \F_q\}$ is the complete set of distinct linear characters of the root subgroup $X_{ab}\leqslant U$. For $\beta\in \F_q$ we define then $f_{ab}^\beta=\frac{1}{q}\sum_{\alpha\in \F_q} \theta_{ab}^\beta (x_{ab}(-\alpha)) x_{ab}(\alpha)\in \C X_{ab}$. Thus $f_{ab}^\beta$ is the primitive idempotent of $\C X_{ab}$ affording $\theta_{ab}^\beta$. Keeping this notation and the notation in \ref{4.7} we have:
\begin{Cor}\label{4.8}
Let $[A]$, $X_{ri}, X_{sj}$ as above, then
\[
f_{ri}^\beta [A]=[A] f_{sj}^{\beta \sigma^{-1}} \text{ for all } \beta \in \F_q.
\]
\end{Cor}
\begin{proof}
For $\beta\in \F_q$ we calculate:
\begin{eqnarray*}
f_{ri}^\beta [A]&=& \frac{1}{q}\sum_{\alpha\in \F_q} \theta_{ri}^\beta(x_{ri}(-\alpha)) x_{ri}(\alpha) [A]\\
&=&  \frac{1}{q}\sum_{\alpha\in \F_q} \theta(-\beta\alpha)[A]x_{sj}(\alpha \sigma)\\
&=& [A] \cdot    \frac{1}{q}\sum_{\alpha\in \F_q}  \theta(-\beta \sigma^{-1}\alpha)x_{sj}(\alpha)\\
&=& [A]f_{sj}^{\beta \sigma^{-1}}, \text{ as desired.}
\end{eqnarray*}
\end{proof}

Recall from \ref{4.5l} that $\Pstab_U^r[A]=U_\cR$, where $\cR$ consists of all positions in zero column and those positions in column containing a main condition, which are on or below the main condition. In the setting of \ref{4.7} let $J=\cR\cup \{(s,j)\}$.

\begin{Lemma}\label{4.9}
$J=\cR\cup \{(s,j)\}$ is a closed subset of $\Phi^+$.
\end{Lemma}
\begin{proof}
Let $(l,s)\in \cR$. Then $(l,s)$ has to be position not above the main condition $(r,s)$ and hence $l\geqslant r$. Recall that we have $1\leqslant j<s<i<r\leqslant n$ by our setting \ref{4.7}, and hence $l>i$ and $(l,j)\in \cR^\circ \subseteq \cR\subseteq J$. Now suppose $(j,t)\in \cR$. We have to show that $(s,t)\in J$. If $t\notin \{j_1, \ldots,j_k\}, \p=\{(i_1, j_1),\ldots, (i_k,j_k)\}$ as above, column $t$ is a zero column in $[A]$ and hence $(s,t)\in \cR^\circ \subseteq \cR\subseteq J$. So let $t=j_\nu$ for some $1\leqslant \nu \leqslant k$. Then $(i_\nu, j_\nu)\in \p$ and we have $j\geqslant i_\nu$, since $(j,t)\in \cR$. Now $s>j$,thus $s>i_\nu$ and hence again $(s,t)=(s, j_\nu)\in \cR^\circ \subseteq \cR\subseteq J$.   
\end{proof}

In fact we have shown that $[X_{sj}, U_\cR]\subseteq U_{\cR^\circ}$. Recall that $\cR=\cR^\circ \cup \p$ by \ref{4.5l} and $U_{\cR^\circ}\trianglelefteq U_\cR$. Thus we have:
\begin{Cor}\label{4.10}
$U_{\cR^\circ}\trianglelefteq U_J$. Moreover $U_J/U_{\cR^\circ}\cong X_{i_1 j_1}\times \cdots \times  X_{i_k j_k} \times X_{sj}$. 
\end{Cor}
Since $U_{\cR^\circ}$ acts trivially on $[A]$ from the right by \ref{3.6}, the linear character $\theta_A$ of \ref{3.6}, henceforth denoted by $\theta_A^r$,  by which $U_\cR$ acts from the right on $[A]$ can be extended to $U_J$ by any linear character $\theta_{sj}^\beta$ of $X_{sj}, \beta\in\F_q$.

\medskip

We set $\Gamma=\{(a, j_\nu)\in \Phi^+\,|\,\nu=1, \ldots, k, j_\nu<a<i_\nu\}$, so $\Phi^+=\cR\dot \cup \Gamma$. We fix a linear ordering on $\Gamma$ such that $(s,j)\in \Gamma$ comes first. Note that by Chevelley's commutator formula
\[\big\{\prod\nolimits_{(a,b)\in \Gamma} x_{ab}(\alpha_{ab})\,|\,\alpha_{ab}\in \F_q,\,\forall\, (a,b)\in \Gamma\big\}\]
is a complete set of right coset representatives of $U_\cR$ in $U$ and is in bijection with the idempotents $[B]\in \O_A^r$, where the products are taken in the given order of $\Gamma$.

\begin{Theorem}\label{4.11}
In the setting of \ref{4.6} keeping the notation introduced above, let $\beta\in \F_q$. Then
\begin{enumerate}
\item [1)] $[A] f_{sj}^\beta u =\theta_A(u)[A] f_{sj}^\beta$ for all $u\in U_\cR$.
\item [2)] $[A]f_{sj}^\beta x_{sj}(\alpha)=\theta(\beta \alpha)[A]f_{sj}^\beta$. Thus $U_J$ acts on $[A]f_{sj}^\beta$ by the linear character $\tilde \theta=\theta_A \times \theta_{sj}^\beta$ with trivial $U_{\cR^\circ}$-action.
\item [3)] $[A]f_{sj}^\beta \C U$ is isomorphic to $\Ind_{U_J}^U \C_{\tilde \theta}$, where $\C_{\tilde \theta}$ is the one dimensional $\C U_J$-module affording the linear character $\tilde \theta$ of $U_J$.
The set $\{[A]f_{sj}^\beta u\,|\, u\in C_{\Gamma'} \}$ is a basis of $[A] f_{sj}^\beta \C U$, where $\Gamma'=\Gamma\setminus \{\(s,j)\}$ and
$
C_{\Gamma'}=\big\{\prod\nolimits_{(a,b)\in \Gamma'} x_{ab}(\alpha_{ab})\,|\,\alpha_{ab}\in \F_q,\,\forall\, (a,b)\in \Gamma'\big\}.
$
In particular $\dim_\C [A]f_{sj}^\beta \C U=q^{a-1}=\frac{1}{q}\dim_\C \C \O_A^r$, where $a=|\Gamma|$.
\end{enumerate}
\end{Theorem}

We remark that theorem \ref{4.11} can be easily extended to idempotents $f_{sj}^\beta$ for more positions $(s,j)\in \hat \cR\setminus \cR$ as long as the corresponding root subgroups commute modulo $U_{\cR^\circ }$. In the next section we shall use the technique of shifting the action of the endomorphism ring $E^r$ from the left to the right in \ref{4.7} to classify all verges $[A]\in \E$ such that $\C \O_A^r$ contains irreducible constituents of minimal dimension and to describe those.

\section{The irreducibles of minimal dimension}

Throughout this section let $[A]\in \E$ be a verge with main conditions $\main[A]=\p=\p_{_A}=\{(i_1, j_1),\ldots,(i_k,j_k)\}\subseteq \Phi^+$. Let $h_\nu=h_{i_\nu j_\nu}, \nu=1,\ldots,k$ and recall from section \ref{4} that $\cL^\circ$ consists of all positions in zero rows together with the positions in rows $i_1,\ldots, i_k$ not to the right of $(i_\nu, j_\nu), \nu=1,\ldots,k$, $\cL=\cL^\circ \cup \p$ and $\hat \cL$ arises by adding to $\cL$ all positions $(r,i)\in \Phi^+$ such that there exist $(i,j), (r,s)\in \p$ with $1\leqslant j<s<i<r\leqslant n$. Then $h_{ij}\cap h_{rs}=\{(i,s)\}\subseteq \Phi^+$ and left action by $X_{ri}$ on $[A]$ fills position $(i,s)$ in $[A]$ and leaves $[A]$ unchanged otherwise. The corresponding sets for the right action are denoted by $\cR^\circ\subseteq \cR^\circ \cup \p=\cR\subseteq \hat \cR$, replacing rows by columns and $(r,i)\in \hat \cL$ by $(s,j)\in \hat \cR$.

\begin{Defn}\label{5.1}
Let $1\leqslant \nu, \mu\leqslant k$. We say $(i_\nu, j_\nu)$ and $(i_\mu, j_\mu)$ are {\bf connected} (in $\p$), if the following holds:
\begin{enumerate}
\item [1)] $j_\nu=i_\mu$ or $j_\mu=i_\nu$ (so  $\nu\not=\mu$ and $h_\nu, h_\mu$ meet at the diagonal).
\item [2)] There exists a hook $h_\rho$, $1\leqslant \rho \leqslant k$, such that $h_\rho$ intersects both, $h_\nu$ and $ h_\mu$, each in precisely one position in $\Phi^+$.
\end{enumerate}
If there are no pairs of connected main conditions in $\p$, we say $\p$ is {\bf hook disconnected}.
\end{Defn}

We illustrate this with main conditions $(i,j), (j,m), (r,s)\in \p$, where $1\leqslant s<j<r<i\leqslant n$, hence $h_{ij}\cap h_{rs}=\{(r,j)\}\not=\emptyset$, and $1\leqslant m<s<j<r\leqslant n$, so $h_ {jm}\cap h_{rs}=\{(j,s)\}\neq \emptyset$. Putting both inequalities together, we obtain  $1\leqslant m<s<j<r<i\leqslant n$.
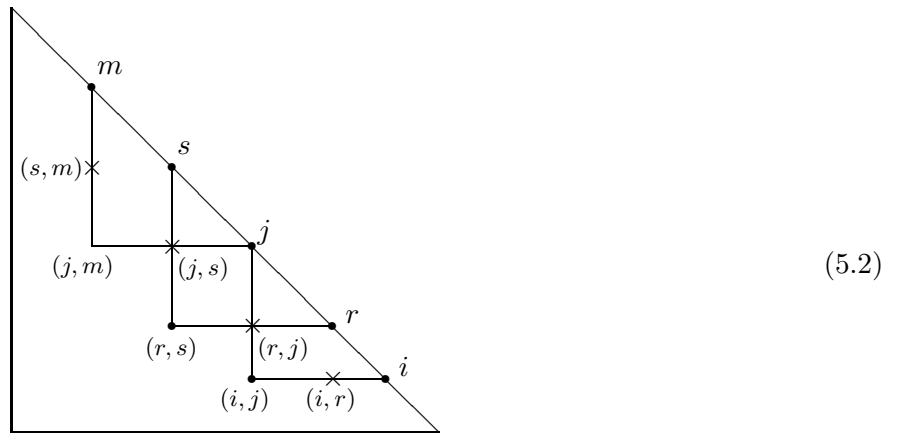
\begin{equation}\label{5.2}
\begin{picture}(120,105)
\put(-40,-60){\line(1,0){160}}
\put(-40,-60){\line(0,1){160}}
\put(-10,10){\line(1,0){60}}
\put(-10,10){\line(0,1){60}}
\put(-40,100){\line(1,-1){160}}
\put(20,-20){\line(1,0){60}}
\put(20,-20){\line(0,1){60}}
\put(50,-40){\line(0,1){50}}
\put(50,-40){\line(1,0){50}}
\put(10,-31){\footnotesize $(r,s)$}
\put(52,-31){\footnotesize $(r,j)$}
\put(38,-50){\footnotesize $(i,j)$}
\put(20,40){\circle*{3}}
\put(20,-20){\circle*{3}}
\put(50,-40){\circle*{3}}
\put(22,45){$s$}
\put(-10,70){\circle*{3}}
\put(-8,75){$m$}
\put(50,10){\circle*{3}}
\put(52,13){$j$}
\put(80,-20){\circle*{3}}
\put(85,-19){$r$}
\put(100,-40){\circle*{3}}
\put(105,-39){$i$}
\put(70,-50){\footnotesize $(i,r)$}
\put(76.1,-42.8){$\times$}
\put(46.1,-22.8){$\times$}
\put(22,-1){\footnotesize $(j,s)$}
\put(-25,0){\footnotesize $(j,m)$}
\put(-37,37){\footnotesize $(s,m)$}
\put(16,7){$\times$}
\put(-14,37){$\times$}
 \end{picture}
\end{equation}

\vskip80pt

We see from (\ref{5.2}) that $(i,r), (r,j)\in \hat \cL\setminus \cL$ but the sum $(i,j)$ of these roots in $\Phi^+$ is contained in $\p$. Thus the commutator group $[X_{ir}, X_{rj}]=X_{ij}$ acts by the nontrivial character $\theta_{ij}^\tau$ from the left (and right) on $[A]$ with $0\neq \tau=A_{ij}\in \F_q$. Similarly $X_{jm}=[X_{js}, X_{sm}]$ acts by $\theta_{jm}^\rho$ on $[A]$ with $\rho=A_{jm}$.

\begin{Theorem}\label{5.3}
Suppose $1\leqslant m<s<j<r<i\leqslant n$ such that $(i,j), (j,m), (r,s)\in \p$. Then $\C \O_A^r$ does not contain irreducible constituents of minimal dimension.
\end{Theorem}
\begin{proof}
Obviously the set $\Delta=\{(i,r), (r,j), (i,j)\}$ is closed in $\Phi^+$, and $(i,r), (r,j)\in \hat \cL\setminus \cL$, whereas $(i,j)\in \p\subseteq \cL$ (compare (\ref{5.2})). Now $U_{\Delta}\leqslant U_{\hat \cL}$ is isomorphic to the unitriangular group $U_3(q)$. Since $X_{ij}=[X_{ir}, X_{rj}]$,  it  acts trivially on every one dimensional $\C U_\Delta$-module.

By \ref{3.11} $\C \O_A^r$ contains irreducible constituents of minimal dimension if and only if $E^r=\End_{\C U} (\C \O_A^r)$ has one dimensional representations. By \ref{4.5} $E^r=\epsilon_{_A} \C U_{\hat \cL}$, where $\epsilon_{_A}\in \C U_{\hat \cL}$ is the central idempotent (see \ref{4.6}) on which $U_{\cL^\circ}$ acts trivially and $U_{\cL}/U_{\cL^\circ}$ by the linear character $\theta_A$ defined in \ref{3.6}. In particular $X_{ij}$ acts on $\epsilon _{_A}$ by the nontrivial character $\theta_{ij}^{A_{ij}}$, since $(i,j)\in \p, A_{ij}\neq 0$. Let $e\in E^r$ be an idempotent affording a one dimensional representation of $E^r$. Since $\epsilon_{_A}$ is the identity of $E^r$, we have $\epsilon_{_A} e=e$. We obtain $e=x_{ij}(\alpha)e=x_{ij}(\alpha) \epsilon_{_A} e=\theta(A_{ij}\alpha) \epsilon_{_A}e=\theta(A_{ij} \alpha) e$, for all $\alpha\in \F_q$ and hence $\theta(A_{ij}\alpha)=1$ for all $\alpha\in \F_q$, a contradiction. Thus $E^r$ has no one dimensional representations and the theorem is proved.
\end{proof}

We now turn to the hook disconnected case.

\begin{Prop}\label{5.4}
Suppose $\p$ is hook disconnected and let $[A]\in \E$ be a verge with $\main[A]=\p_{_A}=\p$. Set $ \hat  \cL ^{-} =\hat \cL\setminus \p$, $ \hat  \cR ^{-} =\hat \cR\setminus \p$. Then $ \hat  \cL ^{-} ,  \hat  \cR ^{-} $ are closed subsets of $\Phi^+$ and
$$U_{\hat \cL}/ U_{\cL^\circ}\cong U_{\hat \cL}/U_{\cL}\times U_{\cL}/U_{\cL^\circ},\quad
U_{\hat \cR}/ U_{\cR^\circ}\cong U_{\hat \cR}/U_{\cR}\times U_{\cR}/U_{\cR^\circ}.$$
\end{Prop}
\begin{proof}
Suppose $(a,b), (b,c)\in \hat   \cL^{-}=$. So $1\leqslant c<b<a\leqslant n$. If $(a,b), (b,c)\in \cL^\circ$, then $(a,c)\in \cL^\circ\subseteq  \hat  \cL ^{-} $, since $\cL^\circ$ is closed in $\Phi^+$. Suppose $(a,b)\in  \hat  \cL ^{-} \setminus\cL^\circ $, but $(b,c)\in \cL^\circ$. Since $U_{\cL^\circ}\trianglelefteq U_{\hat  \cL}$ by \ref{4.5}, we have $x^{-1}y^{-1} x =z \in U_{\cL^\circ}$ for $x\in X_{ab}, y\in X_{bc}\subseteq U_{\cL^\circ}$, and hence $X_{ac}=[X_{ab}, X_{bc}]$ contains $x^{-1} y^{-1}x y=zy\in U_{\cL^\circ}$, proving $(a,c)\in \cL^\circ \subseteq \hat \cL^{-}$. If $(a,b)\in \cL^\circ, (b,c)\notin \cL^\circ$, a similar argument shows $(a,c)\in \cL^\circ \subseteq  \hat  \cL ^{-} $.

Thus suppose $(a,b), (b,c)\notin \cL^\circ$, but both are contained in $ \hat  \cL ^{-} =\hat \cL\setminus \p$. Since $\hat \cL$ is closed by \ref{4.5},  $(a,c)\in \hat \cL$. Thus we have to show $(a,c)\notin\p$. Suppose  $(a,c)\in \p$.  We illustrate the situation:
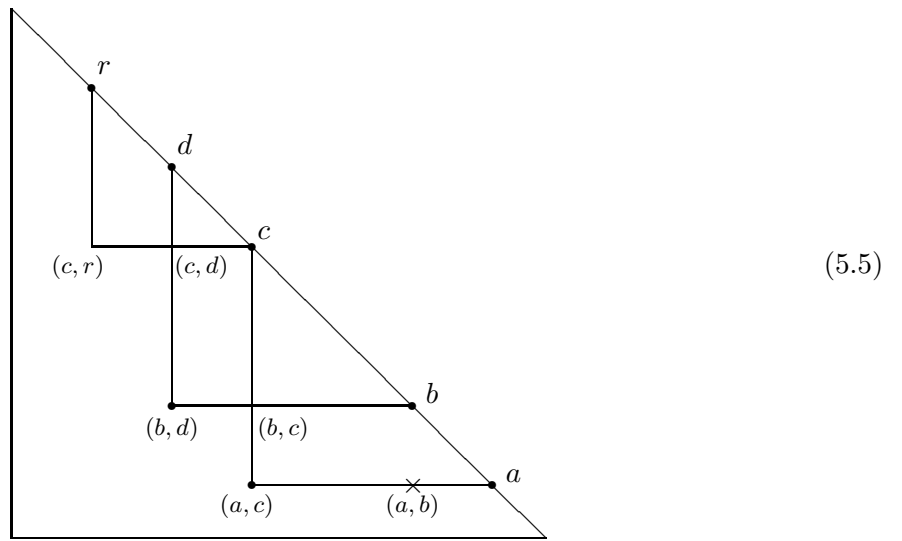
\begin{equation}\label{5.5}
\begin{picture}(120,110)
\put(-40,-100){\line(1,0){200}}
\put(-40,-100){\line(0,1){200}}
\put(-10,10){\line(1,0){60}}
\put(-10,10){\line(0,1){60}}
\put(-40,100){\line(1,-1){200}}
\put(20,-50){\line(1,0){90}}
\put(20,-50){\line(0,1){90}}
\put(50,-80){\line(0,1){90}}
\put(50,-80){\line(1,0){90}}
\put(10,-61){\footnotesize $(b,d)$}
\put(52,-61){\footnotesize $(b,c)$}
\put(38,-90){\footnotesize $(a,c)$}
\put(20,40){\circle*{3}}
\put(20,-50){\circle*{3}}
\put(50,-80){\circle*{3}}
\put(22,45){$d$}
\put(-10,70){\circle*{3}}
\put(-8,75){$r$}
\put(50,10){\circle*{3}}
\put(52,13){$c$}
\put(110,-50){\circle*{3}}
\put(115,-49){$b$}
\put(140,-80){\circle*{3}}
\put(145,-79){$a$}
\put(100,-90){\footnotesize $(a,b)$}
\put(106.1,-82.8){$\times$}
\put(-25,0){\footnotesize $(c,r)$}
\put(21,0){\footnotesize $(c,d)$}
 \end{picture}
\end{equation}

\vskip100pt

Since $(b,c)\notin \p,$ but $(b,c)\in  \hat  \cL ^{-} \setminus \cL^\circ$, there must be a main condition strictly to the left of $(b,c)$ in row $b$, i.e. we find $1\leqslant d <c$ with $(b,d)\in \p$. So $(a,c), (c,r), (b,d)\in \p$ and $h_{ac}, h_{cr}$ intersect $h_{bd}$ at $(b,c)$ and $(c,d)$ respectively. Thus $\p$ is not hook disconnected, a contradiction, hence $(a,c)\in \hat \cL\setminus\p=\tilde \p$.

Thus $ \hat  \cL ^{-} =\cL\setminus \p$ is closed. We have
$\cL ^\circ \subseteq  \hat  \cL ^{-} \subseteq \hat \cL= \hat  \cL ^{-}  \dot \cup \p$, and  $U_{\cL^\circ}\trianglelefteq U_{\hat  \cL}$ by \ref{4.5}. Moreover $U_{ \hat  \cL ^{-} }/U_{\cL^\circ }\leqslant U_{\hat \cL}/U_{\cL^\circ }$ and $U_{ \cL}/U_{\cL^\circ }\cong$ {\huge $\times$}$\!_{(a,b)\in \p}\, X_{ab}$ is a central subgroup of $U_{\hat\cL}/U_{\cL^\circ }$ by \ref{4.5}. Now $U_{ \hat  \cL ^{-} }/U_{\cL^\circ }\cap U_{ \cL}/U_{\cL^\circ }=(1)$, and obviously $(U_{ \hat  \cL ^{-} }/U_{\cL^\circ })\cdot (U_{ \cL}/U_{\cL^\circ })=U_{\hat \cL}/U_{\cL^\circ }$. Since
$U_{ \cL}/U_{\cL^\circ }$ is central in  $U_{\hat \cL}/U_{\cL^\circ }$, we obtain $U_{\hat \cL}/ U_{\cL^\circ}\cong U_{ \hat  \cL  }/U_{\cL}\times U_{\cL}/U_{\cL^\circ}$, as desired. An analogous argument shows the claim for $\cR^\circ \trianglelefteq \hat \cR^{-} \subseteq \hat \cR \subseteq \Phi^+$.
\end{proof}

We remark in passing that $U_{\hat \cL}/U_{\cL}\cong U_{ \hat  \cL ^{-} }/U_{\cL^\circ }$ by the second isomorphism theorem, since $\cL\cap  \hat  \cL ^{-} =\cL^\circ$, $\cL\cup  \hat  \cL ^{-} =\hat \cL$ imply $U_{\cL ^\circ }=U_{\cL}\cap U_{ \hat  \cL ^{-} }$ and $U_{ \hat  \cL ^{-} }U_{\cL}=U_{\hat \cL}$. Obviously $\hat \cL\setminus \cL= \hat  \cL ^{-} \setminus \cL^\circ$ and $\hat \cR\setminus \cR= \hat  \cR ^{-} \setminus \cR^\circ.$

\begin{Cor}\label{5.6}
Suppose $\p\subseteq \Phi^+$ is a hook disconnected set of main conditions and let  $[A]\in \E$ be a verge with $\p_{_A}=\p$. Then $E^r=\End_{\C U} (\C \O_A^r)$ is isomorphic to the group algebra of $U_{ \hat  \cL ^{-} }/U_{\cL^\circ }$.
\end{Cor}
\begin{proof}
By \ref{5.4} $U_{\hat \cL}/ U_{\cL^\circ}\cong U_{ \hat  \cL ^{-} }/U_{\cL^\circ }\times U_{\cL}/U_{\cL^\circ}$. Since $U_{\cL}/U_{\cL^\circ}\cong X_{i_1 j_1}\times \cdots \times X_{i_k j_k}$, $\p=\{(i_1, j_1),    \ldots,$ $ (i_k, j_k)\}$ $\subseteq \Phi^+$ the algebras of the form $(\C U_{\hat \cL}) \epsilon$ are all isomorphic for any primitive idempotent $\epsilon$ in $\c U_\cL$ affording a one dimensional representation with trivial $U_{\cL^\circ}$-action. Choosing $\epsilon$ to be the trivial idempotent $\frac{1}{|U_{\cL}|}\sum_{x\in U_{\cL}} x \in \C U_{\cL}$ yields $E^r=(\C U_{\hat \cL})\epsilon _{_A}\cong \C (U_{\hat \cL}/U_\cL)$, where $\epsilon_{_A}\in \C U_{\cL}$ is defined as in \ref{4.5}. 
\end{proof}

This result has some interesting consequences, which we shall list below. To simplify notation we denote in the situation of \ref{5.4} we identify $U_{ \hat  \cL ^{-} }/U_{\cL^\circ }$  and $U_{\hat \cL}/U_\cL$ and denote this group by $H=H_\p$, so $E^r\cong \C H$.

We define $I\subseteq  \hat  \cL ^{-} $ to consist of $\cL^\circ $ together will all positions $(i,b)\in  \hat  \cL ^{-} $ such that there exist $(i,a), (a,b)\in   \hat  \cL ^{-} \setminus \cL^\circ $. Then $I$ is closed in $\Phi^+$ and $U_I/U_{\cL^\circ}$ is the commutator subgroup $H'=[H, H]$ of $H=U_{ \hat  \cL ^{-} }/U_{\cL^\circ }$. Moreover $H/H'$ is isomorphic to the direct product of the root subgroups $X_{ab}$ with $(a,b)\in \hat  \cL ^{-} $, $(a,b)\notin I$. Let $c=| \hat  \cL ^{-} \setminus I|=| \hat  \cL ^{-} |-|I|$, then $H/H'$ is abelian of  order $q^c$ and hence $H$ has precisely $q^c$ many non isomorphic one dimensional representations. Keeping this notation, we have shown:
\begin{Theorem}\label{5.7}
Let $[A]\in \E$ be a verge with $\main[A]=\p$ then $\C \O_A^r$ has irreducible constituents of minimal dimension $q^{a-b}$, $a=|\O^r_A|,$ $ b=$ number of hook intersections of hooks centered at position in $\p$ if and only if $\p$ is hook disconnected. In this case we have
\begin{enumerate}
\item [1)] $\C \O_A^r$ has precisely $q^c$ many non isomorphic irreducible constituents of minimal dimension $q^{a-b}$, $c=|\hat \cL\setminus I|$ as above, each occurring with multiplicity one in $\c \O_A^r$.
\item[2)] $\End_{\C U}(\C \O_A^r)$ depends only on $\p=\main[A]$ not on the particular values of $A$ at positions $(i,j)\in \p$. Thus varying the nonzero values at positions $(i,j)\in \p$ through $\F_q^*$ one obtains $(q-1)^k q^c$ many non isomorphic irreducible $\C U$-modules which occur in their right orbits with minimal dimension.
\end{enumerate}
\end{Theorem}

Recall the conjectures of Higman  \cite{higman}, Lehrer \cite{lehrer} and Isaacs \cite{Isaacs2} from the introduction. Observe that $q=(q-1)^1+(q-1)^0$, hence $(q-1)^k q^c$ is indeed a polynomial in $(q-1)$ with nonnegative integral coefficients. Summing over all sets $\p\subseteq \Phi^+$ of main conditions, such that $\p$ is hook disconnected, we obtain:
\begin{Theorem}\label{5.8}
There exists a polynomial $d_n(t)\in \Z[t]$ with non negative coefficients such that $d_n(q-1)$ is the number of distinct irreducible characters of $U$, which occur with minimal degree in their supercharacters.
\end{Theorem}

Obviously the Lehrer variant of that statement holds as well.

\section{Monomial sources of irreducibles of minimal dimension}
In this section we shall determine monomial sources of the irreducibles $\C U$-modules of minimal dimension. Here a {\bf monomial source} for an irreducible $\C G$-module $S$, where $G$ is a finite group, means a subgroup $H$ of $G$ together with some one dimensional $\C H$-module $\C _\lambda$ such that $S=\Ind_H^G \C_\lambda$. The index ``$\lambda$'' denotes the linear character of $H$ by which $H$ acts on $\C _\lambda$. Thus $G$ acts monomially on the cosets of $H$ in $G$. A finite group $G$ is called a {\bf monomial group} ($M$-group for short), if every irreducible $\C G$-module is monomial, that is has a monomial source. It is well know (see e.g. Isaacs' book \cite{Isaacsbook}) that every supersolvable and hence every nilpotent group is an $M$-group. Thus in particular finite $p$-groups and hence the unitriangular groups $U_n(q)$ are $M$-groups. Halasi  has shown  in \cite{halasi} that every irreducible character of $U$ is induced from a linear character of
some $\F_q$-algebra subgroup of $U$. In \ref{6.11} we shall prove that for
irreducibles of minimal dimension, the $\F_q$-algebra subgroup can be
chosen to be a pattern subgroup.

We continue with the setting of the previous section. Thus let $\p\subseteq \Phi^+$ be a subset of $\Phi^+$ of main conditions and let $[A]\in \E$ be a verge with main conditions $\p_{_A}=\p$. If not stated otherwise, we assume now that $\p$ is hook disconnected. Recall from the previous sections the definition of the closed subsets $\cL^\circ \subseteq \cL=\cL^\circ \cup \p\subseteq \hat \cL$,  $\cR^\circ \subseteq \cR=\cR^\circ \cup \p\subseteq \hat \cR$ (\ref{4.5} and \ref{4.5l}) and of $ \hat  \cL ^{-} =\hat \cL\setminus \p,  \hat  \cR ^{-} =\hat \cR\setminus \p$ and observe that $| \hat  \cL ^{-} \setminus \cL^\circ|=| \hat  \cR ^{-} \setminus \cR^\circ|=b$ is the number of hook intersections in $\Phi^+$ of hooks centered at main conditions in $\p$. Indeed \ref{3.7} provides bijection:
\begin{equation}\label{6.1}
\perp:\quad  \hat  \cL ^{-} \setminus \cL^\circ \rightarrow  \hat  \cR ^{-} \setminus \cR^\circ : (r,i)\mapsto (r,i)^\perp =(s,j)
\end{equation}
for $(i,j), (r,s)\in \p$ with $1\leqslant j<s<i<r\leqslant n$.

In \ref{4.5l} we showed that $E^r=\End_{\C U} (\C \O_A^r)$ is given as ideal $\epsilon_{_A} \C U_{\hat \cL}$, where $ \epsilon_{_A}$ is the central primitive idempotent in $\C U_{_\cL}$ affording the trivial character on $U_{\cL^\circ}$ and the linear character $\theta_{A}$, defined in \ref{3.6} on  $U_{ \cL}/U_{\cL^\circ }\cong$ {\huge $\times$}$\!_{(a,b)\in \p}\, X_{ab}$.

For $(r,i)\in  \hat  \cL ^{-} \setminus \cL^\circ, (r,i)^\perp =(s,j)\in  \hat  \cR ^{-} \setminus \cR^\circ$, the map:
\begin{equation}\label{6.2}
\Upsilon_{ri}: \, X_{ri}\rightarrow X_{sj}: x_{ri}(\alpha)\mapsto x_{sj}(\alpha \sigma)
\end{equation}
with $\sigma=A_{rs}/A_{ij}\in \F_q^*$ is obviously an isomorphism of abelian groups. In \ref{4.7} we proved that 
\begin{equation}\label{6.3}
x_{ri}(\alpha)[A]=[A]x_{sj}(\alpha \sigma)=[A]\Upsilon_{ri}(x_{ri}(\alpha)).
\end{equation}
Fix a linear ordering on $ \hat  \cL ^{-} \setminus \cL^\circ$ and take all occurring products $\prod_{(a,b)\in \hat \cL\setminus \cL} x_{ab}(\alpha_{ab})$, $\alpha_{ab}\in \F_q$ for $(a,b)\in \hat \cL\setminus \cL$, in that ordering. Transfer this ordering to $ \hat  \cR ^{-} \setminus \cR^\circ$ by $\perp$. Note that $\epsilon_{_A} \C U_{\hat \cL}$ and $\epsilon_{_A}^\perp \C U_{\hat \cR}$ have bases
\begin{equation}\label{6.4}
\big\{\epsilon_{_A}\cdot\prod\nolimits_{(a,b)\in  \hat  \cL ^{-} \setminus \cL^\circ}x_{ab}(\alpha_{ab})\,|\, \alpha_{ab}\in \F_q \text{ for all } (a,b)\in  \hat  \cL ^{-} \setminus \cL^\circ \big\}
\end{equation}
respectively
\begin{equation}\label{6.5}
\big\{\epsilon_{_A}^\perp\cdot\prod\nolimits_{(a,b)\in  \hat  \cR ^{-} \setminus \cR^\circ}x_{ab}(\alpha_{ab})\,|\, \alpha_{ab}\in \F_q \text{ for all } (a,b)\in  \hat  \cR ^{-} \setminus \cR^\circ \big\}.
\end{equation}
thus 
\begin{equation}\label{6.6}
\Upsilon: \epsilon_{_A} \C U_{\hat \cL}\rightarrow\epsilon_{_A}^\perp \C U_{\hat \cR} 
\end{equation}
which sends $\epsilon_{_A}\cdot\prod\nolimits_{(a,b)\in  \hat  \cL ^{-} \setminus \cL^\circ}x_{ab}(\alpha_{ab})$ to  $\epsilon_{_A}^\perp\cdot\prod\nolimits_{(a,b)\in  \hat  \cL ^{-} \setminus \cL^\circ}\Upsilon_{ab}(x_{ab}(\alpha_{ab}))$
in $\epsilon_{_A}^\perp \C U_{\hat \cR}$ is a bijection. We have
\begin{Prop}\label{6.7}
Let $\p$ be hook disconnected. Then $\Upsilon: \epsilon_{_A} \C U_{\hat \cL}\rightarrow\epsilon_{_A}^\perp \C U_{\hat \cR} $ is a $\C$-algebra isomorphism from $E^r=\End_\C U (\C \O_A^r)$ into the subalgebra $\epsilon_A^\perp \C U_{\hat \cR}$ of $\C U$, such that for all $x\in \epsilon_{_A} \C U_{\hat \cL}$ we have:
\[x[A]=[A]\Upsilon(x).\]
\end{Prop}
\begin{proof}
Note that $\epsilon_{_A}$ and $\epsilon_{_A}^\perp$ are central in $\C U_{\hat \cL}$ and $\C U_{\hat \cR}$ respectively by \ref{4.6}.  Moreover the restriction $\Upsilon_{ab}$ of $\Upsilon$ to $X_{ab}$ for $ (a,b)\in  \hat  \cL ^{-} \setminus \cL^\circ$, is a group homomorphism by (\ref{6.2}). Thus by Chevalley's commutator formula (\cite{carter}, 5.2.2),  it suffices to check that $\Upsilon$ preserves commutator relations modulo $U_{\cL^\circ}$. More precisely, using \ref{5.4},
we have to show that $\Upsilon$ respects the commutator relation of the form:
\begin{equation}\label{6.8}
[x_{ar}(\alpha), x_{ri}(\beta)]=x_{ai}(\alpha\beta) \text{ modulo } (1-\epsilon_{_A}) \C U_{\hat \cL}, \,\forall\, \alpha, \beta\in \F_q
\end{equation}
for $(a,r), (r,i)\in  \hat  \cL ^{-} $. So let $1\leqslant s<b<r<a\leqslant n$, $1\leqslant j<s<i<r\leqslant n$ such that  $(a,b),(r,s),(i,j)\in \p$ and $\{(r,b)\}=h_{rs}\cap h_{ab}$, $\{(i,s)\}=h_{ij}\cap h_{rs}$, (compare example \ref{4.4}). Then  $(a,r)^\perp=(b,s)$ and $(r,i)^\perp=(s,j)$.
We distinguish two cases:
\medskip

{\bf Case 1:} Suppose $i<b$. Then $(a,i)$ is to the left of main condition $(a,b)$ in row $a$ and hence contained in $\cL^\circ$. So $X_{ai}\subseteq U_{\cL^\circ}$, and $x_{ai}(\alpha\beta)=(1)\mod (1-\epsilon_{_A})\C U_{\hat \cL}$, since $x_{ai}(\alpha \beta)\epsilon_{_A}=\epsilon_{_A}x_{ai}(\alpha\beta)=\epsilon_{_A}$. Now, since $i<b$, the position $(b,j)$ in column $j$ is below $(i,j)$, and hence $X_{bj}\subseteq U_{\cR^\circ}$. We conclude
\[
[\Upsilon(X_{ar}), \Upsilon(X_{ri})]=[X_{bs}, X_{sj}]=X_{bj}=(1) \mod (1-\epsilon_{_A}^\perp) \C U_{\hat \cR}
\]
proving that $\Upsilon$ preserves commutators in this case.
\smallskip

Note that the case $b=i$ cannot occur, since then $(a,i)=(a,b)\in \p$. Since $\p$ is hook disconnected, $ \hat  \cL ^{-} $ is closed and hence $(a, i)\in   \hat  \cL ^{-} =\hat \cL\setminus \p$.
 \smallskip

{\bf Case 2:} $i>b$. Thus $(a,i)$ is in row $a$ strictly to the right of $(a,b)\in \p$, and hence is not contained in $\cL=\cL^\circ \cup \p$. Moreover $X_{ai}$ acting from the left on $[A]$ will change position $(i,b)$ which is the hook intersection of the hooks $h_{ab}$ and $h_{ij}$. On the other hand $X_{bj}$ acting from the right on $[A]$ changes only position $(i,b)$ in $[A]$ and hence $(a,i)^\perp =(b,j)\in \hat \cR\setminus \cR$. By (\ref{6.2}) $\Upsilon (x_{ai}(\gamma))=x_{bj}(\gamma \frac{A_{ab}}{A_{ij}})$ for all $\gamma\in\F_q$. Similarly,
$\Upsilon(x_{ar}(\alpha))=x_{bs}(\alpha \frac{A_{ab}}{A_{rs}})$ and $\Upsilon(x_{ri}(\beta))=x_{sj}(\beta \frac{A_{rs}}{A_{ij}})$
 and hence
\begin{eqnarray*}
\Upsilon[x_{ar}(\alpha), x_{ri}(\beta)]&=& \Upsilon(x_{ai}(\alpha\beta))=x_{bj}(\alpha\beta \frac{A_{ab}}{A_{ij}})\\&=& x_{bj}(\alpha \frac{A_{ab}}{A_{rs}}\beta \frac{A_{rs}}{A_{ij}})=[x_{bs}(\alpha \frac{A_{ab}}{A_{rs}}), x_{sj}(\beta \frac{A_{rs}}{A_{ij}})]\\
&=& [\Upsilon(x_{ar}(\alpha)), \Upsilon(x_{ri}(\beta))]
\end{eqnarray*}
as desired.
\end{proof}

\begin{Remark}\label{6.9}
For hook connected main conditions $\p$, it can be shown that a similar $\C$-algebra isomorphism $\Upsilon: \epsilon_{_A} \C U_{\hat \cL}\rightarrow\epsilon_{_A}^\perp \C U_{\hat \cR} $ exists. It acts on root subgroups $X_{ab}, (a,b)\in \hat \cL\setminus\cL^\circ$ in the same fashion. However the case $b=i$ in the proof of proposition \ref{6.7} can occur which forces for sequences of hook connected conditions a non trivial action on root subgroups $X_{ij}$ with $(i,j)\in \p$ as well, permuting those, to take care of the case that a commutator subgroup meets root subgroups at main conditions. To keep the character $\theta_A$ stable one needs to adjust the entry $\alpha$ for $x_{ij}(\alpha)$, $(i,j)\in \p$, by factors derived from quotients of entries at main conditions. This extension to the general case however will not be needed in this paper.
\end{Remark}

Let $\p\subseteq \Phi^+$ be again hook disconnected. Recall from \ref{5.4} that the endomorphism rings $E^r, E^l$ of $\C \O_A^r$ and $\C \O_A^l$ of \ref{4.5} and \ref{4.5l} respectively can be identified with the group algebras $\C H$ and $\C H^\perp$, setting $H=H_\p=U_{ \hat  \cL ^{-} }/U_{\cL^\circ}\cong U_{\hat \cL}/U_{\cL}$ and $H^\perp=U_{ \hat  \cR ^{-} }/U_{\cR^\circ}\cong U_{\hat \cR}/U_{\cR}$, where again $ \hat  \cL ^{-} =\hat \cL\setminus \p,  \hat  \cR ^{-} =\hat \cR\setminus \p$ . Obviously the algebra isomorphism $\Upsilon$ of \ref{6.7} induces a group isomorphism $H\rightarrow H^\perp$. For $x\in \C H$ we denote now $\Upsilon (x)\in \C H^\perp$ by $x^\perp$.\\

If $S$ is an irreducible $\C H^\perp$-module, we can extend the $H^\perp$-action on $S$ by \ref{5.4} to $U_{\hat \cR}/U_{\cR^\circ}=H^\perp\times U_{\cR}/U_{\cR^\circ}$ by $\theta_A: U_{\cR}/U_{\cR^\circ}=${\huge $\times$}$_{(a,b)\in\p} X_{ab}\rightarrow \C^*$ defined in \ref{3.6}, and then lift the resulting $U_{\hat \cR}/U_{\cR^\circ}$-module to $U_{\hat \cR}$ by letting $U_{\cR ^\circ}$ act  trivially on it. The corresponding $\C U_{\hat \cR}$-module is now denoted by $\hat S_A=\hat S$. Let $\Irr(\C H^\perp)$ be a complete set of  non isomorphic irreducible $\C H^\perp$-modules.

\medskip

Recall that $|\C \O_A^r|=q^a$, where $a=|\Phi^+\setminus \cR|$, since $U_{\cR}=\Pstab[A]$ by \ref{3.6}. Since $| \hat  \cR ^{-} \setminus \cR^\circ|=|\hat \cR\setminus \cR|=b$ is the number of hook intersections in $\Phi^+$ of hooks centered in $\p$, we have $|\Phi^+\setminus\hat \cR|=a-b$ and hence $[U:U_{\hat \cR}]=q^{a-b}$. Similarly $[U:U_{\hat \cL}]=q^{a-b}$.

\begin{Remark}
In general $\Upsilon$ defined in \ref{6.9} shifts the action on the verge idempotents from left to right. As a consequence all  irreducible constituents of $\C \O_A^r$ are induced from precisely the  irreducible modules of $U_{\hat \cR}$ in the sum of Wedderburn components attached to the central idempotent $\epsilon^\perp_A$. This result has been observed as well by Tung Le in \cite{le} but was shown there by different methods.  In the special case of hook disconnected main conditions it follows immediately from proposition \ref{6.7}:
\end{Remark}

\begin{Cor}\label{6.10}
Let $[A]\in\E$ be a verge such that $\main[A]=\p\subseteq  \Phi^+$ is hook disconnected. Let $\cR^\circ\subseteq \cR\subseteq \hat \cR\subseteq \Phi^+$ be defined as above and $ \hat  \cR ^{-} =\hat \cR\setminus \p, H^\perp=U_{ \hat  \cR ^{-} }/U_{\cR^\circ}\cong U_{\hat \cR}/U_{\cR^\circ}$. If $S$ is an irreducible $\C H^\perp$-module, the lift $\hat S$ of $S$ to $U_{\hat \cR}$ is defined as above. Then we have:
\begin{enumerate}
\item [1)] $\hat S\cong [A]S$ and $[A]S\C U$ is an irreducible constituent of $\C \O_A^r$ of dimension $q^{a-b+m}$ setting $q^m=\dim_\C S$.
\item[2)] $[A] S \C U\cong \Ind_{U_{\hat \cR}}^U [A]  \hat S\cong \Ind_{U_{\hat \cR}}^U \hat S$.
\item[3)] $\C \O_A^r=\sum_{S\in \Irr(\C H^\perp)} (\dim_\C S) [A] S \C U=\bigoplus_{S\in \Irr(\C H^\perp)}(\dim_{\C} S) \Ind_{U_{\hat \cR}}^U \hat S$
is the decomposition of $\C \O_A^r$ into a direct sum of irreducible $\C U$-modules.
\end{enumerate}
\end{Cor}
\begin{proof}
Since $U_{\cR^\circ}$ acts trivially on $[A]$ and $U_{\cR}=\Pstab_U^r[A]$, we conclude that $[A]\C U_{ \hat  \cR ^{-} }=[A]\C H^\perp$ is isomorphic to the regular representation $\C H^\perp_{\C H^\perp}$ of $\C H^\perp$. Since $U_{\cR}/U_{\cR^\circ}$ acts on $[A]$ by the linear character $\theta_A$ we conclude $[A]S\cong [A]\hat S\cong \hat S$. The preimage $\Upsilon^{-1}$ of $S$ under $\Upsilon: \C H\rightarrow \C H^\perp$ is an irreducible $\C H$-module generated by some primitive idempotent $\epsilon_{_S}\in \C H$ say. Thus $\epsilon_{_S} \C H[A]=[A] S$ and hence $[A]S \C U=\epsilon_{_S} [A]\C U=\epsilon_{_S} \C \O_A^r$ is an irreducible constituent of $\C \O_A^r$. Now by Isaacs' theorem (\cite{Isaacsq}, Theorem A) $\dim_\C S$ is a power of $q$, say $q^m$. Thus $\dim_\C  \epsilon_{_S}\C H=q^m$ too, and $\dim_\C \epsilon_{_S} \C \O_A^r=\dim_\C [A] S\C U=q^{a-b+m}$, since $\epsilon_{_S} \C H$ has multiplicity $q^m$ in $\C H_{\C H}$, which may be considered as endomorphism algebra of $\C \O_A^r$. Now 1) follows from \ref{3.10} and \ref{3.10n}. Thus multiplying  up the $\C U_{\hat\cR}$-submodule $[A]S=[A]\hat S$ of $\C \O_A^r$ to $\C U$ produces an increase in dimension by a factor $q^{a-b}$ which incidentally is the index of $U_{\hat \cR}$ in $U$. This implies 2) and hence 3) holds as well.
\end{proof}

In  \ref{6.10} let the irreducible $\C H^\perp$-module be one dimensional and hence afford a linear character $\lambda_S: H\rightarrow \C^*$ which extends to a linear character $\lambda_S\cdot \theta_A: U_{\hat\cR}\setminus U_{\cR^\circ}\rightarrow \C^*$ and hence to a linear character $\hat \lambda_S: U_{\hat \cR}\rightarrow \C^*$. Then we have precisely in the situation of theorem \ref{5.7} and hence have constructed the monomial sources of the irreducible $\C U$-modules which have minimal dimension in their verge orbit module:

\begin{Cor}\label{6.11}
Let $\mathcal S\leqslant \C \O_A^r$ be of minimal dimension. Then $[A]\in \E$ has hook disconnected main conditions $\p\subseteq \Phi^+$.
Moreover there exists a linear character $\lambda: U_{\hat \cR}\rightarrow \C^*$ having $U_{\cR^\circ}$ in its kernel, such that the restriction of $\lambda$ to $U_{\cR}$ is $\theta_A$ and    $(U_{\hat \cR}, \lambda)$ is a monomial source of $\mathcal S$.
\end{Cor}

 This corollary shows in particular that irreducible characters of $U$ of minimal degree in their supercharacters are induced from linear characters of pattern subgroups. Recall that  by \cite{halasi} every irreducible character of $U$ is introduced from a linear character of some $\F_q$-algebra subgroup of $U$. However, as shown by Evseev in \cite{Evseev} not every irreducible character of $U$ needs to be induced from a linear character of some pattern subgroup of $U$. Those, which are, are called ``well-induced'' by Evseev and our result shows that irreducible characters of minimal degree in their supercharacters are well-induced.


\providecommand{\bysame}{\leavevmode ---\ }
\providecommand{\og}{``} \providecommand{\fg}{''}
\providecommand{\smfandname}{and}
\providecommand{\smfedsname}{\'eds.}
\providecommand{\smfedname}{\'ed.}
\providecommand{\smfmastersthesisname}{M\'emoire}
\providecommand{\smfphdthesisname}{Th\`ese}


\begin{thebibliography}{LYLE}
\addcontentsline{toc}{chapter}{Bibliography}

\bibitem{andre1}
{\scshape C. A. M. Andr\'{e}}, {\og Basic characters of the
unitriangular group\fg},  \emph{J. Algebra},  \textbf{175},
(1995), 287-319.


 \bibitem{carter}
 {\scshape R. W. Carter}, {\og Simple groups of Lie type \fg}, John Wiley \& sons, \emph{London, New York, Sydney, Toronto}, (1972).

\bibitem{curtis}
{\scshape C. W. Curtis and I. Reiner}, {\og Methods of representation theory - with applications to finite groups and orders\fg}, \emph{Wiley Classics Library Edition, New York}, \textbf{vol. 2}, (1994).


\bibitem{super}
{\scshape P. Diaconis and I. M. Issacs}, {\og Supercharacters and superclasses for algebra groups\fg}, \emph{Trans. Amer. Math. Soc.},  \textbf{360(5)}, 2359--2392, (2008).


\bibitem{Evseev}
{\scshape A. Evseev}, {\og Reduction for characters of finite
algebra groups\fg},  \emph{J. Algebra},  \textbf{325}, (2011),
321-351.


\bibitem{guo}
{\scshape Q. Guo}, {\og On the $U$-module structure of the unipotent Specht modules of finite general linear groups},  preprint,  arXiv:1304.4370v2.

\bibitem{halasi}
{\scshape Z. Halasi}, {\og On the characters and commutators
of finite algebra groups\fg},  \emph{J. algebra},
\textbf{275}, (2004), 481-487.

\bibitem{higman}
{\scshape G. Higman}, {\og Enumerating $p$-groups I\fg},
\emph{Proc. London Math. Soc. },  \textbf{3}, (1960), 24-30.

\bibitem{huppert}
{\scshape B. Huppert}, {\og A remark on the character-degrees
of some p-groups \fg},  \emph{Arch. Math.},  \textbf{59},
(1992), 313-318.

\bibitem{Isaacsbook}
{\scshape I. M. Isaacs}, {\og Character theory of finite groups\fg},  \emph{Academic Press},  
(1976)

\bibitem{Isaacsq}
{\scshape I. M. Isaacs}, {\og Characters of groups associated with finite algebras\fg} , \emph{ J. Algebra},  \textbf{177}, (1995), 708-730.


\bibitem{Isaacs2}
{\scshape I. M. Isaacs}, {\og Counting characters of upper
triangular groups\fg},  \emph{J. Algebra},  \textbf{315},
(2007), 698-719.


\bibitem{markus}
{\scshape M. Jedlitschky}, {\og Decomposing Andr\'{e}-Neto supercharacters for
Sylow $p$-subgroups of type $D$ \fg}, \emph{PhD. Thesis, Univerit\"{a}t Stuttgart}, 2013


\bibitem{le}
{\scshape Tung Le}, {\og Supercharacters and pattern subgroups in the upper triangular groups\fg},  arXiv:0912.0385v2, (Aug. 6th,  2013).


\bibitem{lehrer}
{\scshape G. Lehrer}, {\og Discrete series and the unipotent
subgroup\fg},  \emph{Comp. Mathematical },  \textbf{28(1)},
(1974), 9-19.


\bibitem{yan}
{\scshape N. Yan}, {\og Representations of finite unipotent linear groups by the method of Clusters\fg}, arXiv:1004.2674v1, (2010)

\end{thebibliography}
\end{document}